\documentclass[a4paper]{article}

\usepackage{amsmath,amssymb,graphicx} 
\usepackage{amsthm}
\numberwithin{equation}{section}
\usepackage[margin=1in]{geometry} 
\usepackage{enumerate}
\usepackage{tikz}
\usepackage{hyperref}
\usepackage{cite}
\usepackage{bm}
\usepackage[flushleft]{threeparttable}
\usepackage{array,booktabs,makecell}
\usepackage{pgfplots}

\newtheorem{theorem}{Theorem}[section]
\newtheorem{definition}{Definition}[section]
\newtheorem{notation}{Notation}
\newtheorem{proposition}{Proposition}[section]

\newtheorem{lemma}{Lemma}[section]
\newtheorem*{lemma*}{Lemma}
\newtheorem{assumption}{Assumption}

\begin{document}

\nocite{*} 

\title{A Mixed Type Generalized Kimura Operator}

\author{Binglu Chen \thanks{Department of Mathematics, University of Chicago, Chicago, IL 60637; {\tt blchen@uchicago.edu}}}

\date{} 

\maketitle

\begin{abstract}
We analyze a class of mixed type generalized Kimura operators on 2-dimensional compact manifolds with corners that find applications in the analysis of topological insulators. We model the operator and  
provide the degenerate H\"older space-type estimates for model operators. With the analysis of perturbation term we establish the existence of solutions.
We also give proofs of the existence and regularity of the global heat kernel.
\end{abstract}

\tableofcontents




\section{Introduction}
This paper is dedicated to a study of a mixed type generalized Kimura diffusion operator. The problem stems from the study of a diffusion operator in \cite{GB} on the triangle $T=\{(x,y):0\leq x,y,x+y\leq 1\}$:
\begin{gather*}
    L=\gamma_{12}\left[xy(\partial_x-\partial_y)^2+(y-x)(\partial_x-\partial_y)\right]\\
    +\gamma_{23}\left[y(x\partial_x+(y-1)\partial_y)^2+(y-1)(x\partial_x+(y-1)\partial_y)\right]\\
    +\gamma_{13}\left[x((x-1)\partial_x+y\partial_y)^2+(x-1)((x-1)\partial_x+y\partial_y)\right].
\end{gather*}

Each of the above three terms is degenerate as a one-dimensional diffusion and they combine to a degenerate operator $L$, which is elliptic inside $T$ but  remains degenerate at the boundary of $T$. In the vicinity of the two boundary segments defined by $x = 0$ and $y = 0$, some diffusion coefficients converge linearly to $0$ while the drift terms converge to a vector field pointing toward the inside of $T$. In the vicinity of the boundary segment $x+y = 1$, the drift term vanishes linearly while the normal components of the diffusion coefficient vanish quadratically.

We generalize this example to a class of operators that we call {\em a mixed type generalized Kimura diffusion operator},
acting on functions defined on 2-dimensional manifolds of corners. 
A mixed type generalized Kimura diffusion operator, $L$ defined on $P$, a two-dimensional compact manifold with corners, is a second-order locally elliptic operator in the interior $\mathring{P}$ with appropriate degeneracy conditions at boundary points. Specifically, in the vicinity of a boundary component, the coefficients of the normal part of the second-order term vanish to order either one or two. 

When the coefficients of the normal part of the second order term vanish exactly to order one along all the boundary components, $L$ is the $\mathit{generalized\ Kimura \ operators}$ as introduced by C. Epstein and R. Mazzeo in \cite{EM}. It is an extension of the Wright-Fisher operator \cite{wright, fisher_1931, 10.1093/genetics/16.2.97}. 

In our setting of mixed degeneracy, as in \cite{EM}, the natural generalized domains of the mixed type generalized Kimura diffusion operator $L$ are 2-
dimensional compact manifolds with corner $P$. A point $p\in bP$, the boundary of $P$, has a relative open neighborhood in $P$, that is homeomorphic to a open neighborhood of $(0,0)$ in $\mathbb{R}_+\times\mathbb{R}$ or $\mathbb{R}^2$, with $p$ mapping to $(0,0)$. We call $p$ an edge point or a corner in these two cases, respectively. Under local adapted coordinates, $L$ is presented in the form given in Definition \ref{definition}.

Let $E$ be a regular edge of $P$, i.e., so that all the points on $E$ are of type \ref{2a}. We say $L$ is tangent to $E$ if at any point $p$ on $E$, the vector field perpendicular to the edge vanishes at $p$, i.e., $d(p)=0$, and $L$ is transverse to $E$ if there exists a $c_E>0$ such that $d(p)>c_E$.

The results in our paper are derived under the following assumption:
\begin{assumption}
 $L$ is either tangent or transverse to any regular edge. 
\end{assumption}

As in \cite{EM}, the second order part of a generalized Kimura operator induces a singular metric on $\mathbb{R}_+^n\times\mathbb{R}^m$:
\begin{gather*}
    ds^2_{\text{WF}}=\sum_{j=1}^n\frac{dx_j^2}{x_j}+\sum_{k=1}^mdy_k^2,
\end{gather*}
\cite{EM} then introduces two families of H\"older weighted spaces induced by this metric. In section \ref{holder weighted space}, we define the metric induced by the second order part of $L$ and two families of H\"older spaces $C^{k,\gamma}(P), C^{k,2+\gamma}(P)$ for $k\in\mathbb{N}_0, 0<\gamma<1$. For each local case of $L$, we solve the heat equation of the corresponding model operator in the model H\"older space. By patching local solutions together and analyzing the corresponding perturbation term, we obtain the first main result of this paper:
\begin{theorem}
For $0<\gamma<1$, if the data $f\in C^{k,2+\gamma}(P),g\in C^{k,\gamma}(P\times[0,T])$, then the inhomogeneous problem 
\begin{gather*}
   (\partial_{t}-L)w=g\ in\ P\times[0,T]\qquad 
   \text{with}\quad  w(0,x,y)=f
\end{gather*}
has a unique solution $w\in C^{k,2+\gamma}(P\times[0,T])$.
\end{theorem}


Based on series expansion of fundamental solutions of model operators and the above result on the heat equation, our second main result is the existence and regularity results of a heat kernel for $L$. We let $P^{reg}$ denote the union of $\mathring{P}$ and regular edge points and regular corner points (see Definition \ref{definition}).
\begin{theorem}
The global heat kernel $H_{t}(d_{1},d_{2},l_{1},l_{2})\in C^{\infty}(P^{reg}\times \mathring{P}\times(0,\infty))$ of the full operator $L$ exists and for $f\in C^0(P)$, then 
\begin{align*}
   v_{f}:=\int_{P}H_{t}(d_{1},d_{2},l_{1},l_{2})f(l_{1},l_{2})dl_{1}dl_{2} 
\end{align*}
is the solution of $(\partial_{t}-L)v_f=0$ with $v_{f}(0,\cdot,\cdot)=f$. 
\end{theorem} 
This heat kernel is smooth in $(d_1,d_2)$ when $(d_1, d_2)\in P^{reg}$. In other words, this includes source contributions at $(d_1,d_2)$ any point on the regular part of $bP$. When $(d_1,d_2)$ is on the infinity edge, we may prove that $H_t(d_1,d_2,l_1,l_2)$ is the product of a delta function on the infinity edge and a one-dimensional heat kernel along that edge; we do not present the details here. The diffusion coefficients vanishing to second order in the normal direction essentially imply that the infinity edge as the name indicates is indeed at infinity in the following sense: any diffusion starting from $\mathring{P}$ would never reach the infinity edge, while a diffusion starting from infinity edge would never enter $\mathring{P}$.

\paragraph{Comparison with previous research.}
There is a rich literature addressing the fundamental solution of the Kimura operator. 
In \cite{shimakura}, Shimakura obtained a closed form of transition probability through computing the the eigenvalues and eigenfunctions of the Wright-Fisher operator. In \cite{wright}, L. Chen and D. Stroock study the Dirichlet heat kernel and its boundary behavior for the one-dimensional Wright-Fisher operator,
\[Lu=x(1-x)u_{xx}, \forall x\in (0,1)\]
through series expansion of the fundamental solutions of the related model operators.

In \cite{epstein2019transition}, C. Epstein and C. Pop studied the martingale problem and semigroup on the weighted $L^2$-space $L^2(P,d\mu)$ associated to the operator, and thus established the transition probabilities and hitting distributions of the diffusion operator. In \cite{harnack}, C. Epstein and R. Mazzeo obtained the Dirichlet heat kernel of the generalized Kimura operator consisting of a single term of the form $k(t,p,q)d\mu_P (q)$, where it is assumed that the weights are strictly positive, i.e. $L$ is transverse to all boundary hypersurfaces. They proved that $k\in C^\infty((0,\infty)\times P \times(P\setminus bP))$ and is uniformly bounded on $P\times P$ for positive times. The main ingredients of their proofs are a scale invariant Poincar\'e inequality established for a large class of weighted Dirichlet forms.

\paragraph{Outline of this paper.}
The plan of this paper is as follows. In section \ref{L} we introduce the mixed type generalized Kimura operator $L$ on 2-dimensional compact manifolds with corner $P$ and the degenerate H\"older space associated with $L$. The operator $L$ is modeled at different boundary points by the model operator $L_{M}$ acting on the model spaces. In section \ref{model_operator} we derive the explicit fundamental solutions of model operators and make a careful analysis of the solution operator in the degenerate H\"older spaces. After this, in section \ref{6} we prove the existence of solutions of the equation \[(\partial_t-L)u=g\ \text{in}\  P\times (0,T],\ u(0, p)=f(p)\]
with data $f,g$ in the degenerate H\"older spaces. In section \ref{7} we establish the existence and regularity results of the heat kernel of $L$.

\section{$L$ on 2-dimensional manifolds with corners}\label{L}
\subsection{Setting}
The natural generalized domains of the mixed type generalized Kimura diffusion operator are 2-dimensional manifolds with corners. Compared to manifolds with or without boundary, which are already universally known in Differential Geometry, manifolds with corners have received comparatively little attention. They were developed by \cite{JC, AD, RB93, RB96} and others in the analysis of elliptic operators on manifolds with corners. A more complete introduction can be found in \cite{DJ}. Although the basic definitions in the subject are not agreed upon, we cite the most commonly used definitions and give a review here.

A paracompact Hausdorff topological space $P$ is a 2-dimensional manifold with corners if for every $p\in P$, there is a neighborhood $\mathcal{U}_{p}$ and a homeomorphism $\psi_{p}$ from $\mathcal{U}_{p}$ to a neighborhood of 0 in $\mathbb{R_{+}^{\mathit{l}}\times\mathbb{R^{\mathit{\mathrm{2}-l}}}}$ for some $l\in\left\{ 0,1,2\right\}$, with $\psi_{p}(p)=0$ and the overlap maps are diffeomorphisms. (Recall that a mapping
between two relatively open sets in $\mathbb{R}^{n}\times \mathbb{R}^{N-n}$ is a diffeomorphism if it is
the restriction of a diffeomorphism between two absolute open sets in $\mathbb{R}^{N}$.) Specifically, if $\mathbf{\psi_{\mathit{p}}:}\ \mathcal{U}_{p}\to\mathcal{V}_{p}$  is the homeomorphism, then for $p\neq q$:
\begin{align*}
  \psi_{p}\circ\psi_{q}^{-1}:\psi_{q}\left(\mathcal{U}_{q}\cap\mathcal{U}_{p}\right)\longrightarrow\psi_{p}\left(\mathcal{U}_{q}\cap\mathcal{U}_{p}\right)
\end{align*}
is a diffeomorphism. If such a map $\psi_{p}$ exists, we say that the point $p$ is an interior point if $l=0$, an edge point if $l=1$, a corner if $l=2$. The codimension $l$ is well defined after imposing smoothness structures. It is due to the fact that the wedges $\left\{(r,\theta):r\geq 0, 0\leq\theta\leq A \right\}$ with various angles (acute angle, $\pi$, obtuse angle, $2\pi$) are different diffeomorphism classes. The definition of manifold with corners excludes the wedge with obtuse angles, hence a non-convex polyhedron appears as the simplest counterexample.

\begin{definition}\label{definition}
Let $P$ be a two-dimensional compact manifold with corners. A second order operator $L$ defined on $P$ is called a generalized Kimura diffusion operator of second kind if it satisfies the following set of conditions:
\begin{description}
\item[1.]$L$ is elliptic in the interior of $P$.
\item[2.]If $q$ is an edge point, then there are local coordinates $(x,y)$ so that in the neighborhood 
\begin{align*}
   \mathcal{U}=\left\{ 0\leq x<1,\  |y|<1\right\} 
\end{align*}
the operator takes one of the following two forms:
\begin{align}
      \label{2a}
      &L=ax\partial_{x}^{2}+bx\partial_{xy}+c\partial_{y}^{2}+d\partial_{x}+e\partial_{y}\tag{2a}\\
     \label{2b}
      &L=ax^{2}\partial_{x}^{2}+bx\partial_{xy}+c\partial_{y}^{2}+dx\partial_{x}+e\partial_{y}.\tag{2b}
\end{align}
We assume that all coefficients $a(x,y), b(x,y),  c(x,y), d(x,y), e(x,y)$ lie in $C^{\infty}(\mathcal{U})$. We call $q$ a regular edge point, infinity edge point, respectively.
\item[3.] If $q$ is a corner, then there are local coordinates $(x,y)$ so that in the neighborhood 
\begin{align*}
    \mathcal{U}=\left\{ 0\leq x<1, \ 0\leq y<1\right\}
\end{align*}
the operator takes one of the following three forms:
\begin{align}
    \label{2c}
    &L=ax\partial_{x}^{2}+bxy\partial_{xy}+cy\partial_{y}^{2}+d\partial_{x}+e\partial_{y}\tag{2c}\\
    \label{2d}
    &L=ax^{2}\partial_{x}^{2}+bxy\partial_{xy}+cy\partial_{y}^{2}+dx\partial_{x}+e\partial_{y}\tag{2d}\\
    \label{2e}
    &L=ax^2\partial_{x}^{2}+bxy\partial_{xy}+cy^{2}\partial_{y}^{2}+dx\partial_{x}+ey\partial_y.\tag{2e}
\end{align}
We assume that all coefficients $a(x,y), b(x,y), c(x,y), d(x,y), e(x,y)$ lie in $C^{\infty}(\mathcal{U})$. We call $q$ a regular regular corner, mixed corner, infinity corner, respectively.
\item[4.]The vector field is inward pointing at edge points of tyep $(2a)$ and corners of type ($2c$), vertical at edge points of $(2b)$, vertical up at corners of type $(2d)$.  
\item[5.] $a(x,y),\ c(x,y)$ are strictly positive on $P$.
\end{description}
\end{definition}

\begin{figure}[htbp] 
\centering 
\includegraphics[width=1.1\textwidth]{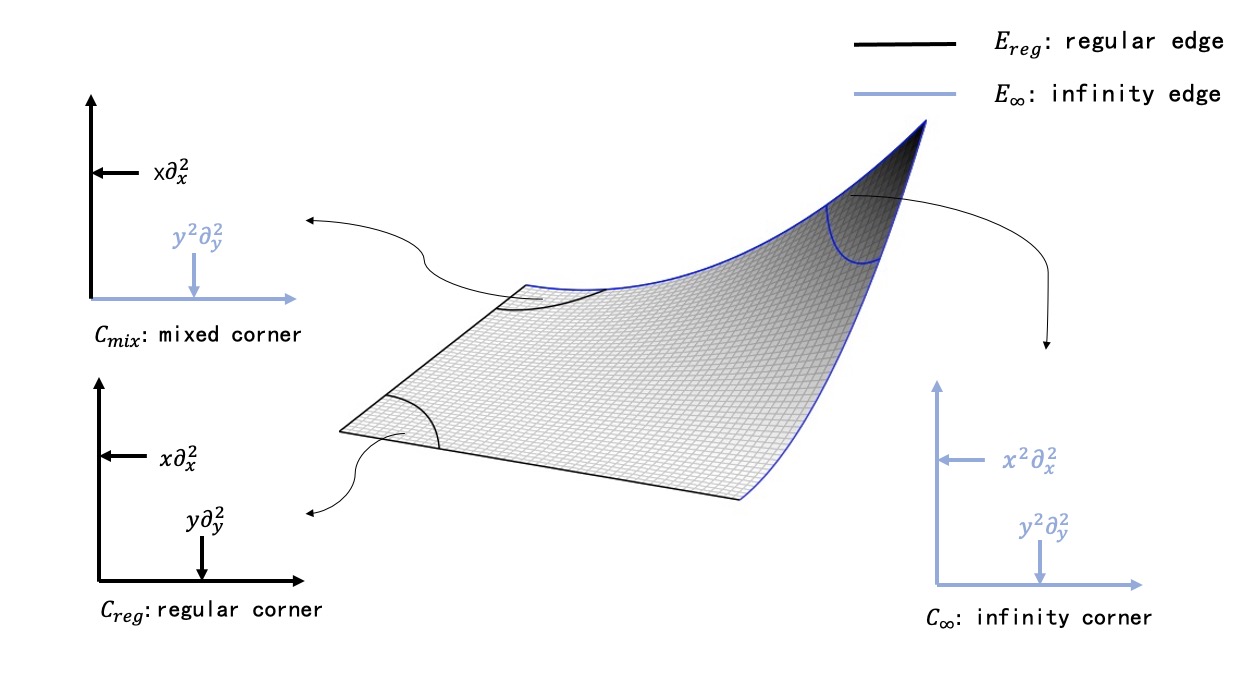} 
\caption{Different types of edges and corners} 
\label{Fig.corner} 
\end{figure}

\newpage

\begin{assumption}\label{assumption}
Let $E$ be a regular edge of $P$. We say
\begin{enumerate}
    \item $L$ is tangent to $E$ if for any point $p$ on $E$, the vector field perpendicular to the edge as locally defined in (\ref{2a}) vanishes, i.e., $d(p)=0$.
    \item $L$ is transverse to $E$ if there exists a $c_E>0$ such that at any point $p$ on $E$, the vector field perpendicular to the edge as locally defined in (\ref{2a}) does not vanish ,
    i.e., $d(p)>c_E$.
\end{enumerate}
We make the assumption that $L$ is either tangent or transverse to any regular edge. In this case we say $p$ is a {\em tangent point} if it lies on one tangent edge, otherwise we say it is a {\em transverse point}.
\end{assumption}

\begin{notation}\label{notation}
In the following we denote
\begin{enumerate}
    \item[\textbullet] $P$: a compact two-dimensional manifold with corners
    \item[\textbullet] $E_{reg}$: the union of regular edge points
    \item[\textbullet] $E_{\infty}$: the union of infinity edge points 
    \item[\textbullet] $C_{reg}$: the set of regular corners 
    \item[\textbullet] $C_{mix}$: the set of mixed corners 
    \item[\textbullet] $C_{\infty}$: the set of infinity corners 
    \item[\textbullet] $P^{reg}$: $\mathring{P}\cup E_{reg}\cup C_{reg}$
\end{enumerate}
\end{notation}

The distinguishing difference between the two forms of $L$ is the vanishing order of  coefficients of the normal part of the second order terms normal to the boundary. It motivates us to introduce a coordinate-invariant definition of $L$.  

We introduce the principal symbol of a differential operator of order $l$. It can be defined being independent of coordinates as a smooth function on the cotangent bundle $T^{*}P$. That is, at each $\xi\in T^{*}_{p}P$, choose a function $f$ such that $df(p) = \xi$ and then set
\begin{equation*}
    \sigma_{l}(D)(p, \xi) = \lim_ {t\to\infty} t^{-l}\left(e^{- itf} \circ D \circ e^{itf}\right)(p). 
\end{equation*}
$\sigma_{l}(D)(p, \xi)$ does not depend on the choice of $f$ is because $\sigma_{l}(D)(p, \xi)$ only depends on the first derivatives of $f$ or equivalently $\xi$.

We can also consider the principal symbol of the second order operator $L$ under the local charts, in the  variables $x,y$:
\begin{align}\label{eq}
    \sigma_{2}\left(L\right)\left(x,y;\xi,\eta\right)=\widehat{a}(x,y)\xi^{2}+\widehat{b}(x,y)\xi\eta+\widehat{c}(x,y)\eta^2.
\end{align}
Here, $\xi$ and $\eta$ are the dual (cotangent) variables associated to $x$ and $y$. We require that $\sigma_{2}\left(L\right)\left(x,y;\xi,\eta\right)$ is smooth, nonnegative for all $(\xi,\eta)$, and strictly positive when $(x,y)$ is not on the boundary and $\left(\mathbf{\mathrm{\mathbf{\xi,\eta}}}\right)\neq\left(0,0\right)$. Furthermore its characteristic set
\begin{align*}
    \mathrm{Char}\left(L\right)=\left\{ \left(x,y;\xi,\eta\right):\sigma_{2}\left(L\right)\left(x,y;\xi,\eta\right)=0\right\}
\end{align*}
 is equal to the set of all conormal vectors to $bP$, or more precisely to the set of all points $(q,v)$ where $q\in bP$ and $v\in T_{q}^{*}P$ vanishes on the tangent spaces of all boundary hypersurfaces which contain q. Finally we require that $\sigma_{2}\left(L\right)$ vanishes on conormal vectors to precisely first order on the edge of type $(a)$ and to second order on the edge of type $(b)$. 

Conversely, in a local coordinate system, at a point $q$ in the interior of the edge $x=0$, the conormal is spanned by the vector $v=dx$. If we write $\sigma_{2}\left(L\right)\left(x,y;\xi,\eta\right)$ as a quadratic form as above, because $\sigma_{2}\left(L\right)\left(q;\xi,\eta\right)$ only vanishes on the conormal vectors,  this gives that $\widehat{a}(q)=\widehat{b}(q)=0,\ \widehat{c}(q)\neq 0$. If the vanishing is simple, then $\widehat{a}(x,y)=xa(x,y),\ \widehat{b}(x,y)=xb(x,y)$. If the vanishing is quadratic, then $\widehat{a}(x,y)=x^2a(x,y),\ \widehat{b}(x,y)=xb(x,y)$. Then at the intersection of two edges, the corner should have three kinds of forms and the forms are given naturally.

\subsection{H\"older Weighted Space}\label{holder weighted space}
In \cite[II.5]{EM}, Epstein and Mazzeo introduce anisotropic H\"older spaces, which are similar to the usual H\"older spaces, but are defined by a singular metric on $\mathbb{R}_+^n\times\mathbb{R}^m$ induced by the Kimura diffusion operator:
\begin{gather*}
    ds^2_{\text{WF}}=\sum_{j=1}^{n}\frac{dx_j^2}{x_j}+\sum_{k=1}^{m}dy_k^2.
\end{gather*}
Motivated by this, we first define the local singular metric induced by the operator then define the associated H\"older spaces induced by the local metric. We denote $\mathbb{R}_+:=[0,\infty)$.

\begin{definition}
We define the metric spaces
\begin{enumerate}
    \label{s_c_reg}
    \item $S_{c\_reg}$: $(\mathbb{R}_+^2,d)$ equipped with the norm
    \[d((x_{1},x_{2}),(x'_{1},x'_{2}))=2|\sqrt{x_{1}}-\sqrt{x'_{1}}|+2|\sqrt{x_{2}}-\sqrt{x'_{2}}|.\]
    
    \label{s_c_mix}
    \item $S_{c\_mix}$: $(\mathbb{R}_+^2,d)$ equipped with the norm    
    \[d((x,y),(x',y'))= 2|\sqrt{x}-\sqrt{x'}|+|\text{ln} y-\text{ln} y'|.\]
    
    \label{s_c_infty}
    \item $S_{c\_\infty}$: $(\mathbb{R}_+^2,d)$ equipped with the norm    
    \[d((y_1,y_2),(y_1',y_2'))= |\text{ln} y_1-\text{ln} y_1'|+|\text{ln} y_2- \text{ln} y_2'|.\]
    
    \label{s_e_reg}
    \item $S_{e\_reg}$: $(\mathbb{R}_+\times\mathbb{R},d)$ equipped with the norm
    \[d((x,y),(x',y'))=2|\sqrt{x}-\sqrt{x'}|+|y-y'|.\]
    
    \label{s_e_infty}
    \item $S_{e\_\infty}$: $(\mathbb{R}\times\mathbb{R}_+,d)$ equipped with the norm    
    \[d((y_1,y_2),(y_1',y_2'))= |y_1-y_1'|+|\text{ln}y_2- \text{ln} y_2'|.\]
\end{enumerate}
\end{definition}
We use the $(x,y)$ notation above: $L$
is Kimura in the direction of $x$, and is elliptic or quadratic in the direction of $y$. To unify the notation, we make the following convention of first tangential derivative of $y$:
\begin{gather*}
D_yu=
\begin{cases} 
\begin{array}{cc}
\partial_yu & L\ \text{is elliptic in}\ y, i.e. \ref{s_e_reg}\\
y\partial_yu &  L\  \text{is quadratic in}\ y, i.e. \ref{s_c_mix}, \ref{s_c_infty}, \ref{s_e_infty}
\end{array}.
\end{cases}
\end{gather*}

\begin{definition}
Let $S$ represents one of the metric spaces above. We denote 
\begin{enumerate}
    \item $\dot{C}^k(S)$: the closure in $C^k(S)$ of compactly supported smooth functions
    
    \item $\dot{C}^{k,2}(S)$:  the closure in $C^{k}(S)$ of compactly supported smooth functions with respect to the norm:
    \begin{gather*}
        ||f||_{k,2}:=||f||_{C^{k-1}}+\underset{|\alpha|+|\beta|=k}{\sup}||(\mathbf{\partial_x})^\alpha(\mathbf{D_y})^{\beta}f||_2,\\
        ||f||_{2}:=||f||_{\infty}+||(\partial_xf, \mathbf{D_y}f)||_{\infty}+\underset{|\alpha|+|\beta|=2}{\sum}||(\mathbf{\sqrt{x}\partial_x})^{\alpha}(\mathbf{D_y})^{\beta}f||_{\infty}.
    \end{gather*}
\end{enumerate}
For $0<\gamma<1$, set the norm
\begin{gather*}
    [f]_{\gamma}:=\underset{\mathbf{x}\neq\mathbf{x'}}{sup}\frac{|f(\mathbf{x})-f(\mathbf{x'})|}{d(\mathbf{x},\mathbf{x'})^\gamma},\\
    [f]_{2+\gamma}:=[f]_{\gamma}+[(\mathbf{\partial_x}f, \mathbf{D_y}f)]_\gamma+\underset{|\alpha|+|\beta|=2}{\sum}[(\mathbf{\sqrt{x}\partial_x})^{\alpha}(\mathbf{D_y})^\beta f]_\gamma.
\end{gather*}
The space $C^{k,\gamma}(S)$,  $C^{k,2+\gamma}(S)$ are the subspace of $\dot{C}^{k}(S)$, $\dot{C}^{k,2}(S)$ consisting of functions $f$ for which the norm 
\begin{gather*}
    ||f||_{k,\gamma}:= || f ||_{C^k} + \underset{|\alpha|+|\beta|=k}{sup} [\mathbf{\partial^{\alpha}_{x}}\mathbf{D^\beta_y}f]_{\gamma},\\
    ||f||_{k,2+\gamma}:= || f ||_{C^{k}} + \underset{|\alpha|+|\beta|=k}{sup} [\mathbf{\partial^{\alpha}_{x}}\mathbf{D^\beta_y}f]_{2+\gamma}.
\end{gather*}
are finite, respectively.
\end{definition}

Similarly we define the parabolic H\"older spaces.

\begin{definition}
We denote
\begin{enumerate}
    \item $\dot{C}^{k,\frac{k}{2}}(S\times [0,T])$: the closure in $C^{k,\frac{k}{2}}(S\times [0,T])$ of compactly supported smooth functions

    \item $\dot{C}^{k+2,\frac{k}{2}+1}(S\times [0,T])$: the closure in $C^{k,\frac{k}{2}}(S\times [0,T])$ of compactly supported smooth functions with respect to the norm:
    \begin{gather*}
        ||f||_{k+2,\frac{k}{2}+1}:=||f||_{C^{k,\frac{k}{2}}}+\underset{|\alpha|+|\beta|+ 2|j|=k}{sup}||(\mathbf{\partial_{x}})^{\alpha}(\mathbf{D_y})^{\beta}\partial^j_tf||_{2,1},\\
        ||f||_{2,1}:=||f||_{\infty}+||(\partial_{t}f,\mathbf{\partial_{x}}f,\mathbf{D_y}f)||_{\infty}+\underset{|\alpha|+|\beta|=2}{sup}\||(\mathbf{\sqrt{x}\partial_{x}})^{\alpha}(\mathbf{D_y})^{\beta}f||_{\infty}.
    \end{gather*}
\end{enumerate}
For $0<\gamma<1$, set the norm 
\begin{align*}
    [f]_{\gamma}:=\underset{(t,\mathbf{x})\neq(s, \mathbf{x'})}{sup}\frac{|f(t,\mathbf{x})-f(s,\mathbf{x'})|}{\left(d(\mathbf{x},\mathbf{x'})+\sqrt{|t-s|}\right)^\gamma},\\
    [f]_{2+\gamma}:=[f]_\gamma+[(\partial_{t}f,\mathbf{\partial_x}f,\mathbf{D_y}f)]_\gamma+\underset{|\alpha|+|\beta|=2}{\sum}[(\mathbf{\sqrt{x}\partial_{x}})^{\alpha}(\mathbf{D_y})^{\beta}f]_{\gamma}.
\end{align*}
The space $C^{k,\gamma}(S\times [0,T])$, $C^{k,2+\gamma}(S\times [0,T])$ are the subspaces of $\dot{C}^{k,\frac{k}{2}}(S\times [0,T])$, $\dot{C}^{k+2,\frac{k}{2}+1}(S\times [0,T])$ consisting of functions $f$ for which the norm 
\begin{gather*}
    ||f||_{k,\gamma}:=|| f ||_{C^{k,\frac{k}{2}}} + \underset{|\alpha|+|\beta|+2j=k}{sup} [\mathbf{\partial^{\alpha}_{x}}\mathbf{D^{\beta}_y}\partial^{j}_{t}f]_{\gamma},\\
    ||f||_{k,2+\gamma}:=|| f ||_{C^{k,\frac{k}{2}}} + \underset{|\alpha|+|\beta|+2j=k}{sup} [\mathbf{\partial^{\alpha}_{x}}\mathbf{D^{\beta}_y}\partial^j_tf]_{2+\gamma}
\end{gather*}
are finite, respectively. 
\end{definition}

\section{Model Operator}\label{model_operator}
\subsection{Fundamental Solutions}
\paragraph{1. Analysis of $L_{M}$ at $C_{reg}$}
(\cite{EM} Proposition 3.0.1) If the boundary point $q\in C_{reg}$, then there is a neighborhood $U_{q}$ of $q$ and smooth local coordinates $(x_1,x_2)$ centered at $q$, in terms of which $L$ takes the form 
\begin{align}\label{c_reg}
L=x_1\partial_{x_1}^2+x_2\partial_{x_2}^2+a(x_1,x_2)\partial_{x_1x_2}+d(x_1,x_2)\partial_{x_1}+e(x_1,x_2)\partial_{x_2}.
\end{align}
We introduce the model operator in the space $S_{c\_reg}$:
\begin{align}\label{c_reg1}
    L_{M}=x_{1}\partial_{x_1}^2+x_{2}\partial_{x_{2}}^2+d\partial_{x_{1}}+e\partial_{x_{2}}
\end{align}
where $d=d(0,0), e=e(0,0)$.
In this case $L_{M}$ is the generalized Kimura diffusion operator which has been studied in
\cite{EM}. The solution kernel of $L_{M}$ is the product of two one-dimensional Kimura operator kernel: 
\begin{align}\nonumber
    K_{t}(x_1,x_2,x'_{1},x'_2)=p^{d}_t(x_1,x'_{1})p_t^e(x_2,x'_2)
\end{align}
where $p^{d}_t(x_1,x'_{1})$ satisfies $p_{t}=xp_{xx}+dp_{x}$, and is explicitly given by \cite{CW19}(38): 
if $d>0$:
\begin{gather}
    p^d_t(x,x')=\frac{1}{t}e^{-\frac{x+x'}{t}}(\frac{x}{x'})^{\frac{1-d}{2}}I_{d-1}\left(2\sqrt{\frac{xx'}{t^{2}}}\right)\\
    =(\frac{x'}{t})^{d}e^{-\frac{x+x'}{t}}\psi_{d}(\frac{xx'}{t^{2}})\frac{1}{x'}
\end{gather}
if $d=0$,
\begin{align}\nonumber
    p^d_t(t,x,x')=\frac{1}{t}e^{-\frac{x+x'}{t}}(\frac{x}{x'})^{\frac{1}{2}}I_{1}\left((2\sqrt{\frac{xx'}{t^2}}\right)+e^{-\frac{x}{t}}\delta(x')
\end{align}
here $I_{\alpha}$ is the modified Bessel function of the first kind that satisfies $x^{2}y''+xy'-(x^{2}+\alpha^{2})y=0$

\paragraph{2. Analysis of $L_{M}$ at $E_{reg}$}\label{4} 
If the boundary point $q\in E_{reg}$, then there is a neighborhood $U_{q}$ of $q$ and smooth local coordinates $(x,y)$ centered at $q$, in terms of which $L$ takes the form 
\begin{align}\label{ereg}
L=x\partial_x^2+a(x,y)\partial_y^2+b(x,y)\partial_{xy}+d(x,y)\partial_x+e(x,y)\partial_y
\end{align}
where $a(0,0)=1$.
We introduce the model operator in the space $S_{e\_reg}$:
\begin{align}\nonumber
    L_{M}=x\partial^2_x+\partial^2_y+d\partial_x.
\end{align}
The solution kernel of $L_{M}$ is
\begin{align*}
    K_{t}(x,y,x_{1},y_{1})=p^d_t(x,x_{1})k^e_t(y,y_1)
\end{align*}
where $k_t^e(y,y_1)=\frac{1}{\sqrt{4\pi t}}e^{-\frac{(y-y_1)^2}{4t}}$ is the heat kernel.

\paragraph{3. Analysis of $L_{M}$ at $C_{mix}$}\label{3} 
If the boundary point $q\in C_{mix}$, then there is a neighborhood $U_{q}$ of $q$ and smooth local coordinates $(x,y)$ centered at $q$, in terms of which $L$ takes the form 
\begin{align}\label{cmix}
L=a(x,y)x\partial_x^2+b(x,y)y^2\partial^2_y+c(x,y)xy\partial_{xy}+d(x,y)\partial_{x}+e(x,y)y\partial_{y}
\end{align}
where $a(0,0)=1$.
We introduce the model operator in the space $S_{c\_mix}$:
\begin{align}\nonumber
    L_M=x\partial_x^2+b\partial_y^2+d\partial_x+by\partial_y
\end{align}
where $b=b(0,0)$.
The solution kernel of $L_{M}$ is
\begin{align*}
    K_{t}(x,y,x_{1},y_{1})=p^d_t(x,x_{1})k^{e'}_{bt}(y,y_1)
\end{align*}
where $k_{bt}^{e'}(y,y_1)=\frac{1}{\sqrt{4\pi b t}}e^{-\frac{(\text{ln} y-\text{ln} y_1)^2}{4bt}}\frac{1}{y_1}$.

\paragraph{4. Analysis of $L_M$ at $E_\infty$}
If the boundary point $q\in E_{\infty}$, then there is a neighborhood $U_{q}$ of $q$ and smooth local coordinates $(x,y)$ centered at $q$, in terms of which $L$ takes the form 
\begin{align}\label{einfty}
L=a(y_1,y_2)\partial_{y_1}^2+b(y_1,y_2)y_2^2\partial^2_{y_2}+c(y_1,y_2)y_1y_2\partial_{y_1y_2}+d(y_1,y_2)\partial_{y_1}+e(y_1,y_2)y_2\partial_{y_2}
\end{align}
where $a(0,0)=1$.
We introduce the model operator in the space $S_{e\_\infty}$:
\begin{align}\nonumber
    L_M=\partial_{y_1}^2+by_2^2\partial_{y_2}^2+by_2\partial_{y_2}
\end{align}
where $b=b(0,0)$.
The solution kernel of $L_{M}$ is
\begin{align*}
    K_{t}(y_1,y_2,y'_1,y'_2)=k^e_t(y_1,y'_1)k^{e'}_{bt}(y_2,y'_2).
\end{align*}

\paragraph{5. Analysis of $L_M$ at $C_\infty$}
If the boundary point $q$ is in $C_\infty$, then there is a neighborhood $U_{q}$ of $q$ and smooth local coordinates $(x,y)$ centered at $q$, in terms of which $L$ takes the form 
\begin{align}\label{cinfty}
L=a(y_1,y_2)y_1^2\partial_{y_1}^2+b(y_1,y_2)y_2^2\partial_{y_2}^2+c(y_1,y_2)y_1y_2\partial_{y_1y_2}+d(y_1,y_2)y_1\partial_{y_1}+e(y_1,y_2)y_2\partial_{y_2}.
\end{align}
We introduce the model operator in the space $S_{c\_\infty}$:
\begin{align}\nonumber
    L_M=ay_1^2\partial_{y_1}^2+by_2^2\partial_{y_2}^2+ay_1\partial_{y_1}+by_2\partial_{y_2}
\end{align}
where $a=a(0,0), b=b(0,0)$.
The solution kernel of $L_{M}$ is
\begin{align*}
    K_{t}(y_1,y_2,y'_{1},y'_{2})=k^{e'}_{at}(y_1,y'_1)k^{e'}_{bt}(y_2,y'_2).
\end{align*}

\begin{definition}
The coordinates and forms of $L$ introduced above are called \textbf{local adapted coordinates} and \textbf{local normal forms centered at $q$}.
\end{definition}

\subsection{H\"older Estimates} 
Let $S$ be one of the spaces in $C_{reg}, C_{mix}, C_{\infty}, E_{reg}, E_{\infty}$. 
\begin{proposition}\label{cauchy}
Let $k\in\mathbb{N}$, $R>0$, and $0<\gamma<1$, assume that $f\in C^{k,\gamma}(S)$ and $f$ is supported in $B^{+}_{R}(\textbf{0})$. The solution $v$ to
\begin{align}\label{cauchy_boundary}
    (\partial_{t}-L_{M})v(t,x,y)=0,\ v(0,x,y)=f
\end{align}
belongs to $C^{k,\gamma}(S\times[0,T])$, and there is a constant $C_{k,\gamma,R}$ so that
\begin{align*}
    ||v||_{k,\gamma}\leq C_{k,\gamma,R}|| f||_{k,\gamma}.
\end{align*}
If $f\in\mathcal{C}^{k,2+\gamma}(S)$, then 
\begin{align*}
    ||v||_{k,2+\gamma}\leq C_{k,\gamma,R}|| f||_{k,2+\gamma}.
\end{align*}
\end{proposition}
\begin{proof}
At $C_{reg}$ and $E_{reg}$, these are known results in \cite{EM} Proposition 10.1.1. 
At $C_{mix}$, 
\[L_M=x\partial_x^2+d\partial_x+y^2\partial^2_y+y\partial_y.\]
Through a coordinate change $y\mapsto \text{ln} y$, it is equivalent to consider
\[L_M=x\partial_x^2+d\partial_x+\partial^2_y\]
$\text{\ in}\ \mathbb{R}_+\times\mathbb{R}$ equipped with the induced metric $||\cdot||'_{k,\gamma}, ||\cdot||'_{k,2+\gamma}$. This metric is slightly different from (\ref{s_e_reg}) in that the $k-$partial derivative is well defined at $y=-\infty$ and when restricted to $y=-\infty$ satisfies the corresponding h\"older continuous conditions with respect to $x$.

Let $f\in C^{k,\gamma}(\mathbb{R}_+\times\mathbb{R})$, we first decompose $f$ as \begin{align*}
    f=f_{1}+f_{2}=f_{1}+f(x, -\infty)\chi(y)
\end{align*}
where $\chi\in\mathcal{C}_c^\infty([-\infty,\infty)),\ \chi(-\infty)=1$ such that $||f_2||'_{k,\gamma}\leq 2||f||'_{k,\gamma}$.
Then $f_1\in C^{k,\gamma}(\mathbb{R}_+\times\mathbb{R})$ vanishes at infinity, for which we can apply \cite{EM} Proposition 10.1.1 again to show that there exists a solution \[v_1=\int_{\mathbb{R}_+\times\mathbb{R}} p^d_t(x,x_1)k^e_t(y,y_1)f_1(x_1, y_1)dx_1dy_1\in C^{k,\gamma}(\mathbb{R}_+\times\mathbb{R}\times [0,\infty))\]
with initial data $f_1$ such that 
\begin{align*}
    ||v_1||_{k,\gamma}\leq C||f_1||_{k,\gamma}=C||f_1||'_{k,\gamma}\leq 3C||f||'_{k,\gamma}.
\end{align*}
When $k=0$, it is not hard to show that $v_1, \partial_yv_1, \partial_{yy}v_1$ can be continuously extend to $y=-\infty$ by 0. When $k>0$, using the formula
\begin{gather*}
    \partial^\alpha_x\partial^\beta_y\partial^j_tv_1=\int_{\mathbb{R}_+\times\mathbb{R}} p^{d+\alpha}_t(x,x_1)k^e_t(y,y_1)(\partial^\alpha_{x_1}\partial^\beta_{y_1}L^j_{d+\alpha,M})f_1(x_1, y_1)dx_1dy_1,
\end{gather*}
we can show that for $\alpha+\beta+2j=k, t\geq 0$, $v_1,  \partial_y(\partial_x^\alpha\partial^\beta_y\partial^j_t)v_1, \partial_{yy}(\partial_x^\alpha\partial^\beta_y\partial^j_t)v_1$ can be continuously extend to $y=-\infty$ by 0. Thus $(\partial_t-L_M)v_1=0, v_1(0,\cdot)=f_1$ holds up to $y=-\infty$ and
\begin{align*}
    ||v||'_{k,\gamma}=||v||_{k,\gamma}\leq  3C||f||'_{k,\gamma}.
\end{align*}

This leaves the Cauchy problem with initial data $f_2$. We begin by writing
\begin{gather*}
    K_tf_2(x,y)=\int_S K_t(x,y,x_1,y_1)f(x_1,0)\chi(y_1)dx_1dy_1\\
    =\left(\int_0^\infty p^d_t(x,x_1)f(x_1,0))dx_1\right)\left(\int_{-\infty}^\infty k_t^e(y,y_1)\chi(y_1)dy_1\right)=:I_1(t,x)\cdot I_2(t,y).
\end{gather*}
Since $f(x,0)\in C^{k,\gamma}(\mathbb{R}_+), \chi(y)\in C^{k,\gamma}(\mathbb{R})$, then $I_1(t,x)\in  C^{k,\gamma}(\mathbb{R}_+\times [0,T]), I_2(t,y)\in C^{k,\gamma}(\mathbb{R}\times [0,T])$ are solutions to one-dimensional Cauchy problem respectively. Moreover, since
\begin{gather*}
    \underset{y\to -\infty}{\lim}I_2(t,y)=\chi(-\infty)=1,
\end{gather*} 
$K_tf_2$ can be continuous extended to  $y=-\infty$ by $K_tf_2(x,-\infty)=I_1(t,x)$. Moreover for $\alpha+\beta+2j=k, t>0$, $\partial_y(\partial^\alpha_x\partial^\beta_y\partial^j_t)K_tf_2, \partial_{yy}(\partial^\alpha_x\partial^\beta_y\partial^j_t)K_tf_2$ can also be continuously extended to $y=-\infty$ by 0.

So $K_tf_2\in C^{k,\gamma}(\mathbb{R}_+\times\mathbb{R}\times [0,T])$ is the solution to the Cauchy problem with initial data $f_2$ such that
\begin{align*}
    ||K_tf_2||'_{k,\gamma}\leq C||f_2||'_{k,\gamma}\leq 3C||f||'_{k,\gamma}.
\end{align*}
In all, let $v=v_1+K_tf_2$, then $v$ satisfies (\ref{cauchy_boundary}) up to the quadratic boundary $y=0$ and 
\[||v||_{k,\gamma}\leq 6 C||f||_{k,\gamma}.\] 
Particularly along the edge $y=0$, $v(t,x,0)=I_1(t,x)$ and satisfies
\begin{gather*}
    (\partial_t-x\partial_x^2-d\partial_x)v(t,x,0)=0\ \text{with}\  v(0,x,0)=f(x,0).
\end{gather*}

If $f\in C^{k,2+\gamma}(\mathbb{R}_+\times\mathbb{R})$, this shall be established similarly. For the remaining case when $S=E_{reg}, E_{\infty}$, the proofs are essentially the same as above and we don't give details here.
\end{proof}

\begin{proposition}\label{inhomogeneous}
Let $k\in\mathbb{N}$, $R>0$, and $0<\gamma<1$. Assume that $g\in C^{k,\gamma}(S\times [0,T])$ and $g$ is supported in $B^{+}_{R}(0)\times [0,T]$. The solution $u$ to
\begin{align}\label{inhomo_problem}
    (\partial_{t}-L_{M})u(t,x,y)=g(t,x,y),\ u(0,x,y)=0
\end{align}
belongs to $C^{k,2+\gamma}(S\times[0,T])$, and there is a constant $C_{k,\gamma,R}$ so that
\begin{align*}
    ||u||_{k,2+\gamma}\leq C_{k,\gamma,R}(1+T)||g||_{k,\gamma}.
\end{align*}
The tangential first derivatives satisfy a stronger estimate: there is a constant $C$ so that if $T\leq 1$, then
\begin{align}\label{tangent_y}
   ||D_yu||_{\gamma,T}\leq CT^{\frac{\gamma}{2}}||u||_{\gamma,T}.
\end{align}
\end{proposition}
\begin{proof}
At $C_{reg}$ and $E_{reg}$, these are known results in \cite{EM} Proposition 10.2.1 . 
At $C_{mix}$, 
\[L_M=x\partial_x^2+d\partial_x+y^2\partial^2_y+y\partial_y.\]
Through a coordinate change $y\mapsto \text{ln} y$, we may view
\[L_M=x\partial_x^2+d\partial_x+\partial^2_y\ \ \text{\ in}\ \mathbb{R}_+\times\mathbb{R}\]
equipped with the induced metric $||\cdot||'_{k,\gamma}, ||\cdot||'_{k,2+\gamma}$.

Let $g\in C^{k,\gamma}(S\times [0,T])$. We first decompose $g$ as 
\begin{align*}
    g=g_1+g_2=g_1+g(t,x,0)\chi(y)
\end{align*}
where $\chi\in\mathcal{C}_c^\infty([-\infty,\infty)),\ \chi(-\infty)=1$ such that $||g_2||'_{k,\gamma}\leq 2||g||'_{k,\gamma}$.
Then $g_1\in C^{k,\gamma}(\mathbb{R}_+\times\mathbb{R}\times [0,T])$ vanishes at infinity, for which we can apply \cite{EM} Proposition 10.2.1 to show that there exists a solution $u_1\in C^{k,2+\gamma}(\mathbb{R}_+\times\mathbb{R}\times [0,T])$ with initial data 0 such that 
\begin{gather*}
    ||u_1||_{k,2+\gamma}\leq C(1+T)||g_1||_{k,\gamma}=C(1+T)||g_1||'_{k,\gamma}\leq 3C(1+T)||g||'_{k,\gamma}\\
    ||\partial_yu||_{k,\gamma, T}\leq CT^{\frac{\gamma}{2}}||u||_{\gamma,T}.
\end{gather*}
It is not hard to show that for $\alpha+\beta+2j=k, t\geq 0$, $u, \partial_y(\partial^\alpha_x\partial^\beta_y\partial^j_t)u, \partial_{yy}(\partial^\alpha_x\partial^\beta_y\partial^j_t)u$ can be continuously extend to $y=-\infty$ by 0, so
\begin{gather*}
    ||u_1||'_{k,2+\gamma}=||u_1||_{k,2+\gamma}\leq 3C(1+T)||g||'_{k,\gamma}\\
    ||\partial_yu||'_{k,\gamma, T}\leq CT^{\frac{\gamma}{2}}||u||'_{\gamma,T}.
\end{gather*}
and $(\partial_t-L_M)u_1=g_1, u_1(0,\cdot)=0$ holds up to $y=-\infty$.

This leaves the inhomogeneous problem with $g_2$. We begin by writing
\begin{gather*}
    A_tg_2=\int_0^t\int_S K_{t-s}(x,y,x_1,y_1)g(s,x_1,0)\chi(y_1)dx_1dy_1ds\\
    =\int_0^t\left(\int_0^\infty p^d(t-s,x,x_1)g(s,x_1,0))dx_1\right)\left(\int_{-\infty}^\infty k_{t-s}^e(y,y_1)\chi(y_1)dy_1\right)ds\\
    =\int_0^t I_1(t,s,x)\cdot I_2(t,s,y)ds.
\end{gather*}
If $g(t,x,0)\in C^{k,\gamma}(\mathbb{R}_+\times [0,T]), \chi(y)\in C^{k,\gamma}(\mathbb{R}\times [0,T])$, then $\int_0^tI_1(t,s,x)ds\in  C^{k,2+\gamma}(\mathbb{R}_+\times [0,T])$, $\int_0^t I_2(t,s,y)ds\in C^{k,2+\gamma}(\mathbb{R}\times [0,T])$ are solutions to one-dimensional inhomogeneous problem respectively. So
\begin{align*}
    (\partial_t-L_M)A_tg_2=g(t,x,0)\chi(y)+\int_0^t(\partial_t-L_M)I_1(t,s,x)\cdot I_2(t,s,y)ds=g(t,x,0)\chi(y).
\end{align*}
and 
\begin{gather}\label{partial_y_estimate_1}
   |\partial_yA_tg_2|\leq \sqrt{T}||g_2||_\infty.
\end{gather}
In the following we first assume that $k=0$, $A_tg_2, \partial_xA_tg_2, x\partial^2_xA_tg_2$ can be continuously extended to  $y=-\infty$ since $\underset{y\to -\infty}{\lim}I_2(t,s,y)=\chi(-\infty)=1$. 
Next we verify the h\"older continuity condition. 
\begin{gather*}
    |x\partial_x^2A_tg_2-x'\partial_x^2A_tg_2|\leq \int_0^t |x\partial_x^2I_1(t,s,x)-x'\partial_x^2I_1(t,s,x)|\cdot I_2(t,s,y)ds\\
    \leq \int_0^t |x\partial_x^2I_1(t,s,x)-x'\partial_x^2I_1(t,s,x)|ds\leq C(1+T)|\sqrt{x}-\sqrt{x'}|^\gamma||g(t,x,0)||_{k,\gamma}.
\end{gather*}
Similarly 
\begin{gather}\label{partial_y_estimate_2}
    |\partial_yA_tg_2-\partial_y'A_tg_2|\leq \int_0^t |\partial_yI_2(t,s,y)-\partial_y'I_2(t,s,y)|\cdot I_1(t,s,y)ds\\
    \leq\int_0^t |\partial_yI_2(t,s,y)-\partial_y'I_2(t,s,y)|ds\leq CT^{\frac{\gamma}{2}}|y-y'|^\gamma||\chi||_{k,\gamma}.
\end{gather}
When $k>0$, for $\alpha+\beta+2j=k$, using the formula
\begin{gather*}
    (\partial^\alpha_x\partial^\beta_y\partial^j_t)A_tg_2=\int_0^t\int_{\mathbb{R}_+\times\mathbb{R}}p^{d+\alpha}_{t-s}(x,x_1)k^e_s(y,y_1)(\partial^\alpha_{x_1}\partial^\beta_{y_1}L^j_{d+\alpha,M})g_2(x_1,y_1)dx_1dy_1ds+\sum_{r=0}^{j-1}\partial^{j-r-1}_t\partial^\alpha_x\partial^\beta_yL^r_{d+\alpha,M}g_2,
\end{gather*}
so $A_tg_2, \partial_x(\partial^\alpha_x\partial^\beta_y\partial^j_t)A_tg_2, x\partial^2_x(\partial^\alpha_x\partial^\beta_y\partial^j_t)A_tg_2$ can be continuously extended to  $y=-\infty$ and h\"older continuity conditions hold.

Above all $A_tg_2\in C^{k,2+\gamma}(\mathbb{R}_+\times\mathbb{R}\times [0,T])$ is the solution to the inhomogeneous problem with $g_2$ such that
\begin{align*}
    ||A_tg_2||'_{k, 2+\gamma}\leq C(1+T)||g_2||'_{k,\gamma}\leq 3C(1+T)||g||'_{k,\gamma}.
\end{align*}
Let $u=u_1+A_tg_2$, then $u$ solves the inhomogeneous problem (\eqref{inhomo_problem}) and $||u||_{2+\gamma}\leq 6C(1+T)||g||_{k,\gamma}$. (\ref{tangent_y}) follows from (\ref{partial_y_estimate_1}), (\ref{partial_y_estimate_2}). Particularly along the edge $y=0$, $v(t,x,0)=\int_0^tI_1(t,s,x)ds$ and satisfies
\begin{gather*}
    (\partial_t-x\partial_x^2-d\partial_x)v(t,x,0)=g(x,0)\ \text{with}\  v(0,x,0)=0.
\end{gather*}
This concludes the proof for $S=C_{mix}$.
Taken together, we obtain the conclusion. The remaining case $S=E_{reg}, E_{\infty}$ shall be treated similarly.
\end{proof}

\section{Existence of Solution}\label{6} 
We now return to our principal goal, namely the analysis of $L$ defined on $P$. The estimates proved in the previous sections allow us to prove existence of a unique solution to the inhomogeneous problem 
\begin{gather}\label{25}
   (\partial_{t}-L)w=g\ in\ P\times[0,T]\\
   \text{with}\ w(0,x,y)=f.
\end{gather}

\begin{definition}
Let $\mathfrak{M}=\{(W_j,\phi_j): j=1,\cdot\cdot\cdot,K\}$ be a cover of $bP$ by normal coordinate charts, $W_0\subset\subset int\ P$, covering $P\backslash\cup_{j=1}^KW_j$ and let $\{\varphi_j: j=0,\cdot\cdot\cdot,K\}$ be a partition of unity subordinate to this cover. A function $f\in C^{k,\gamma}(P)$ provided $(\varphi_jf)\circ\phi_{j}\in \mathcal{C}^{k,\gamma}(W_j)$ for each j. We define a global norm on $\mathcal{C}^{k,\gamma}(P)$ by setting 
\begin{align*}
    ||f||_{k,\gamma}=\sum_{j=0}^{K}||(\varphi_{j}f)\circ\phi_{j}||_{k,\gamma}^{W_j}.
\end{align*}
\end{definition}

\begin{theorem}\label{heat_equation}
For $0<\gamma<1$, if the data $f\in C^{k,2+\gamma}(P),g\in C^{k,\gamma}(P\times[0,T])$, then equation (\ref{25}) has a unique solution $w\in C^{k,2+\gamma}(P\times[0,T])$.
\end{theorem}

Uniqueness of the solution can be obtained from the following maximum principle.

\begin{proposition}\label{max prin}(Maximum Principle) 
Let $u\in C^{k,2+\gamma}(P)$ be a subsolution of $\partial_{t}u\leq Lu$ on $P\times[0,T]$, then 
\begin{align}
\underset{P\times[0,T]}{\max}u(t,x,y)=\underset{P}{\max}\ u(0,x,y).
\end{align}
\end{proposition}

\begin{proof}
We show that if $u(0,x,y)\leq 0$, then $u(t,x,y)\leq 0$ for $0<t\leq T$. The standard argument (see \cite{Fri}) shows that the maximum of $u$ on $P\times[0,T]$ occurs along the distinguished boundary $P\times\left\{ t=0\right\} \cup\left\{bP\times[0,T]\right\}$. The regularity hypotheses of $u$ show that the maximum cannot occur along regular edges $E_{reg}\times(0,T]$. Hence it is sufficient to show that $u(t,x,y)<0$ for $(0, T]\times E_{\infty}$.

For a point $q$ on infinity edge $E_{\infty}$, by definition, there exist local coordinates $(x,y)$ centered at $q$ and a neighborhood $\mathcal{U}=\{0\leq x <1,|y|<1\}$ such that $L$ has the form
\begin{align*}
    L= ax^{2}\partial^{2}_{x}+bx\partial_{xy}+c\partial^{2}_{y}+V.
\end{align*}
After a coordinate map $z=-\text{ln} x$, $L$ transforms to a uniform elliptic operator $L'$ on $D=[0,\infty)\times[0,1]$. Let $v(t,x,z)=u(t,x,e^{-z})$. Then $(\partial_{t}-L')v\leq 0$ on $D$. Because $u\in C^{0}(P\times [0,T])$, $v$ is bounded on $[0,\infty)\times[0,1]$, so there exists $A,a>0$ such that $v\leq Ae^{a(x^{2}+z^{2})}$.

We use the fundamental solution of the parabolic equation constructed in \cite{FG} and let the initial point to be $(0,0)$ to see that there exists a barrier function $U_{\tau}$ on $[0,\tau)_{t}\times \mathbb{R}^{2}$ satisfying 
\begin{align*}
    (\partial_{t}-L')U_{\tau}=0
\end{align*}
and there exists $C,G>0$ such that 
\begin{align*}
    U_{\tau}\geq \frac{C}{\tau-t}e^{\frac{G(x^{2}+z^{2})}{4(\tau-t)}}.
\end{align*}

For $\epsilon_{1}, \epsilon_{2}>0$, set
\begin{align*}
    v_{\tau}(t,x,z)= v(t,x,z)-\epsilon_{1}U_{\tau}(t,x,z)+\epsilon_{2}\frac{1}{1+t},
\end{align*}
then $v_{\tau}$ is a strict subsolution of $(\partial_{t}-L')v_{\tau}<0$. We choose $\tau$ so that $\frac{G}{4\tau}>a$. Let
\begin{align*}
    D_{R}= (0,\tau'] \times (D\cap bB_{R}(0)),
\end{align*}
$B_{R}(0)$ is the ball of radius of $R$ around the origin in $\mathbb{R}^{2}$ and $0<\tau'<\tau$. Then with $R$ sufficiently large, $v_{\tau}$ is negative on $D_{R}$. Let $\epsilon_{1}, \epsilon_{2}\to 0$, we see $v(t,x,z)\leq 0$ on the infinity, therefore $u(t,q)\leq 0$ for $0<t\leq \tau'$.
Therefore we proved that
\begin{align*}
    \underset{P\times[0,T]}{\max}u(t,x,y)\leq\underset{P}{\max}\ u(0,x,y).
\end{align*}
\end{proof}

\subsection{Parametrix construction}\label{6.1} 


Let $P$ denote the two-dimension compact manifold with corners, $C_{r}=\{C_{r,i}\}$ the set of regular corners, $C_{m}=\{C_{m,i}\}$ the set of mix corners, $C_{\infty}=\{C_{\infty,i}\}$ the set of infinity corners, $E_{r}=\{E_{r,i}\}$ the set of regular edges, $E_{\infty}=\{E_{\infty,i}\}$ the infinity edges, then
\begin{align*}
  bP=C_{r}\bigcup C_{m}\bigcup C_{\infty}\bigcup E_{r}\bigcup E_{\infty}. 
\end{align*}
First we fix a function $\varphi_{U}\in C^{\infty}(P)$  
that is equal to 1 in a neighborhood of $bP$. Let $U$ be a neighborhood of $bP$ such that $bU\cap intP$ is a smooth hypersurface in $P$, and $\bar{U}\subset\subset\varphi_{U}^{-1}(1)$. The subset $P_{U}=P\cap U^{c}$ is a smooth compact manifold with boundary and $L\mid_{P_{U}}$ is a non-degenerate elliptic operator.
\begin{align*}
    P=U\bigsqcup P_{U},\bar{U}\subset\subset\varphi_{U}^{-1}(1)
\end{align*}

We can double $P_{U}$ across its boundary to obtain $\widetilde{P_{U}}$, which is a compact manifold without boundary. And the operator $L$ can be extended to a classical elliptic operator $\widetilde{L}$ on $\widetilde{P_{U}}$. 

The classical theory of non-degenerate parabolic equations on compact manifolds without boundary, applies to construct an exact solution operator
\begin{align*}
    u_{i}=\widetilde{Q^{t}}[(1-\varphi_U)g]
\end{align*}
to the inhomogeneous equation:
\begin{gather*}
    (\partial_{t}-\tilde{L})u_{i}=(1-\varphi_U)g\ \text{in}\ \widetilde{P_{U}}\times[0,T]\\
    \text{with}\ u_{i}(0,p)=0,\ p\in\widetilde{P_{U}}.
\end{gather*}
This operator defines bounded maps from $C^{k,\gamma}(\widetilde{P_{U}}\times[0,T])\to C^{k,2+\gamma}(\widetilde{P_{U}}\times[0,T])$ for any $0<\gamma<1,k\in\mathbb{N}$.
We set interior parametrix
\begin{align}
    \widehat{Q_{i}^{t}}=\psi\widetilde{Q^{t}}[(1-\varphi_U)g]
\end{align}
where we choose $\psi\in C_{c}^{\infty}(P_{U})$ so that $\psi\equiv1$ on a neighborhood of the support of $(1-\varphi_U)$.

\begin{figure}[ht] 
\centering 
\includegraphics[width=1.1\textwidth]{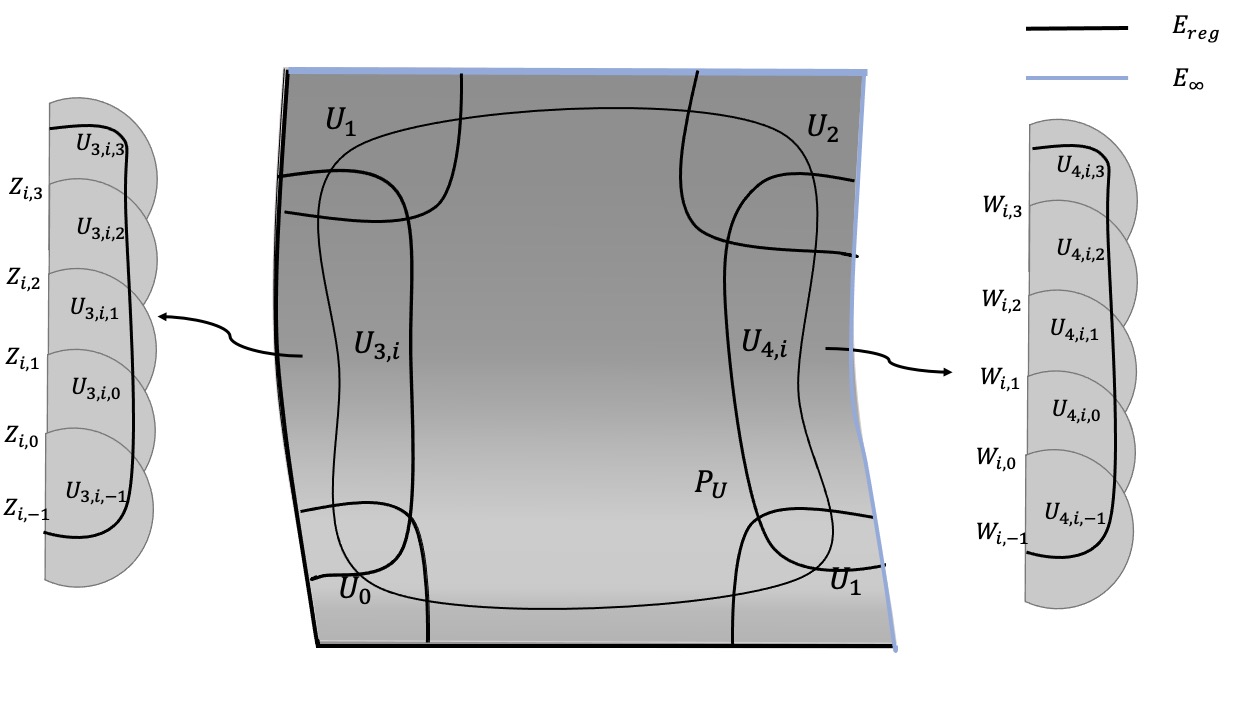} 
\caption{Covering in the proof} 
\label{Fig.cover} 
\end{figure}

To build the boundary parametrix, we construct the $\epsilon$-grid as follows. 
Let 
\[\mathfrak{U}=\{U_{0,i},U_{1,i}, U_{2,i}, U_{3,i}, U_{4,i}\}\] 
be an NCC covering of $bP$, where $U_{0,i},U_{1,i}$, $U_{2,i}$ are $5\epsilon$-neighborhoods of $i$-th regular corner, $i$-th mix corner, $i$-th infinity corner, respectively. By shrinking the neighborhoods we can assume that $U_{0,i},U_{1,i}, U_{2,i}$ are disjoint.

All of the charts in $U_{3,i}, U_{4,i}$ have coordinates lying in $\mathbb{R}_{+}\times\mathbb{R}$. We let $Z_{i,j}, W_{i,j}$ be the points in $U_{3,i},U_{4,i}$ with coordinates \[\{Z_{i,j}=(0,\epsilon j):j\in\mathbb{Z}\}, \{W_{i,j}=(0,\epsilon j):j\in\mathbb{Z}\}\]
respectively. And we let $U_{3,i,j},\ U_{4,i,j}$ to be the $5\epsilon$-neighborhood of $Z_{i,j},W_{i,j}$. To simplify notation we do not keep track of the dependence on $\epsilon$. We illustrate the covering in Fig \ref{Fig.cover}.

Let $(x,y)$ denote normal cubic coordinates in one of these neighborhoods, $U_{n,i,j}$ (if $n=0,1,2$, assume $j=0$). In these coordinates the operator $L$ takes the normal form $L_{i,j}$. We let $L_{n,i,j,M}$ be the corresponding model operator and let $A_{n,i,j}^{t}$ denote the solution operator for the model problem
\begin{align*}
  (\partial_{t}-L_{n,i,j,M})u=g,u(x,0)=0.  
\end{align*}

For $n=3,4$, let $\chi_{n}\equiv 1$ in $(0,2)\times(-2,2)$ and vanishing outside $(0,3)\times(-3,3)$, $\phi_{n}\equiv1$ in $(0,4)\times(-4,4)$ and vanishing outside $(0,5)\times(-5,5)$, and $\chi_n,\phi_n\in C_{c}^{\infty}(\mathbb{R}_{+}\times\mathbb{R})$, respectively. We define 
\[\widetilde{\chi_{n,i,j}}(x,y)=\chi_{n}(\frac{x}{\epsilon^2},\frac{y-\epsilon j}{\epsilon}), \phi_{n,i,j}(x,y)=\phi_{i}(\frac{x}{\epsilon^2},\frac{y-\epsilon j}{\epsilon}).\]
For $n=0,1,2$, we let
\begin{align*}
\widetilde{\chi_{n,i,j}}, \widetilde{\phi_{n,i,j}}=
\begin{cases}
\begin{array}{cc}
   \chi_{n}(\frac{x_1}{\epsilon^2},\frac{x_2}{\epsilon^2}), \phi_{n}(\frac{x_1}{\epsilon^2},\frac{x_2}{\epsilon^2})  & n=0 \\
   \chi_{n}(\frac{x}{\epsilon},\frac{y}{\epsilon^2}),\phi_{n}(\frac{x}{\epsilon},\frac{y}{\epsilon^2}) & n=1\\
   \chi_n(\frac{y_1}{\epsilon},\frac{y_2}{\epsilon}), \phi_n(\frac{y_1}{\epsilon},\frac{y_2}{\epsilon}) & n=2
\end{array}
\end{cases}
\end{align*}
where $\chi_{n}\equiv 1$ in $(0,2)\times(0,2)$ vanish outside $(0,3)\times(0,3)$, $\phi_{n}\equiv1$ in $(0,4)\times(0,4)$ vanish outside $(0,5)\times(0,5)$, and $\chi_{n},\phi_{n}\in C_{c}^{\infty}(\mathbb{R}^{2}_{+})$, respectively.

By abuse of notation, we still denote $\widetilde{\chi_{n,i,j}},\phi_{n,i,j}$ to indicate their pullback to $P$. It is clear that there exists $S>0$, independent of $\epsilon$, such that for $r\in P$,
\begin{align*}
    \chi(r)=\sum_{i=0}^{4}\underset{i,j}{\sum}\widetilde{\chi_{n,i,j}}(r)\leq S.
\end{align*}
It is clear that $\chi_{\epsilon}(r)\geq1$ for $r\in bP$. We can arrange to have $\varphi_{U}=1$ on the set $\chi_{\epsilon}\geq\frac{1}{2}$ and $Supp\ \varphi_{U}\subset\chi_{\epsilon}^{-1}([\frac{1}{16},S])$. Based on that, $\mathfrak{U'}=\{U_{n,i,j}\}$ together with $P_{U}$ is a partition of $P$.

To get a partition of unity of a neighborhood of $bP$ subordinate to $U_{i,j}$, we replace the functions $\{\widetilde{\chi_{n,i,j}}\}$ with
\begin{align*}
    \chi_{n,i,j}=\varphi_{U}\cdot\left(\frac{\widetilde{\chi_{n,i,j}}}{\sum_{n=0}^4\underset{i,j}{\sum}\widetilde{\chi_{n,i,j}}}\right).
\end{align*}
For each $\epsilon>0$, we define a boundary parametrix by setting 
\begin{align}\label{boundary parametrix}
    \widehat{Q_{b,\epsilon}^{t}}=\sum_{n=0}^{4}\underset{i,j}{\sum}\phi_{n,i,j}A_{n,i,j}^{t}\chi_{n,i,j}.
\end{align}
In total we set 
\begin{align}\label{30}
    \widehat{Q^{t}}g=\widehat{Q_{b,\epsilon}^{t}}g+\widehat{Q_{i}^{t}}g.
\end{align}

\subsection{Perturbation Estimate} 
Next we are going to analyze the perturbation term, 
\begin{gather*}
    (\partial_{t}-L)\widehat{Q_{b,\epsilon}^t}g =(\partial_t-L)\left(\sum_{n=0}^{4}\underset{i,j}{\sum}\phi_{n,i,j}A_{n,i,j}^{t}\chi_{n,i,j}g\right)\\
    =\left(\sum_{i=0}^{4}\underset{i,j}{\sum}\chi_{n,i,j}g\right)+\left(\sum_{i=0}^{4}\underset{i,j}{\sum}\phi_{n,i,j}(L_{n,i,j,M}-L)A_{n,i,j}^{t}[\chi_{n,i,j}g]\right)+\left(\sum_{i=0}^{4}\underset{i,j}{\sum}[\phi_{n,i,j},L]A_{n,i,j}^{t}[\chi_{n,i,j}g]\right)\\
    =\varphi_{U}g+E_{\epsilon}^{0}(g)+E_{\epsilon}^{1}(g)
\end{gather*}
where 
\begin{gather*}
    \varphi_{U}=\sum_{n=0}^{4}\underset{i,j}{\sum}\chi_{n,i,j},\\ E_{\epsilon}^{0}(g)=\sum_{n=0}^{4}\underset{i,j}{\sum}\phi_{n,i,j}(L_{n,i,j,M}-L)A_{n,i,j}^{t}[\chi_{n,i,j}g]\\
    E_{\epsilon}^{1}(g)=\sum_{n=0}^{4}\underset{i,j}{\sum}[\phi_{n,i,j},L]A_{n,i,j}^{t}[\chi_{n,i,j}g].
\end{gather*}
And 
\begin{gather*}
    (\partial_{t}-L)\widehat{Q_i^{t}}g=(1-\varphi_{U})g+[\psi,L]\widehat{Q_i^{t}}[(1-\varphi)g]=(1-\varphi_{U})g+E_i^\infty(g).
\end{gather*}
Taken together,
\begin{gather}\label{perturbation}
    (\partial_{t}-L)\widehat{Q^{t}}g =g+E_{\epsilon}^{0}(g)+E_{\epsilon}^{1}(g)+E_i^\infty(g).
\end{gather}

As $\phi\equiv 1$ on the support $\chi$, the norm of $E_\epsilon^1$ is bounded and by $Ce^{-\frac{M}{T}}$ as $T\to 0$ for some $M>0$ by \cite{EM} Proposition 10.3.1.

Then we estimate $E_\epsilon^0$.
\subsubsection{n=0} 
In $U_{0,i}$, under local adapted coordinates, the operator $L$ takes the form
\begin{gather*}
    L=x\partial_{xx}+y\partial_{yy}+xyb(x,y)\partial_{xy}+ d(x,y)\partial_{x}+e(x,y)\partial_{y}.\\
    E_{0,\epsilon}^0=\phi_\epsilon(L_M-L)A^t[\chi_\epsilon g]=-\phi_\epsilon(xyb(x,y)\partial_{xy}+\tilde{b}\partial_x+\tilde{e}\partial_y)A^t[\chi g].
\end{gather*}
Since $\phi_\epsilon$ is supported in the set where $x,y\leq 5\epsilon^2$, we have
\begin{align}\label{estimate1}
    ||E_{0,\epsilon}^0||_\infty\leq C\epsilon^2||\chi_\epsilon g||_{0,\gamma}\leq C\epsilon^{2-\gamma}||g||_{0,\gamma}.
\end{align}
We used 
\[[fgh]_\gamma\leq ||fg||_\infty[h]_\gamma+||fh||_\infty[g]_\gamma+||gh||_\infty[f]_\gamma\]
and $[\phi_\epsilon]_\gamma\leq C\epsilon^{-\gamma}$, so
\begin{align}\label{estimate2}
    [E_{0,\epsilon}^0]_{\gamma}\leq C\epsilon^{2-\gamma}||\chi_\epsilon g||_\gamma\leq C\epsilon^{2-2\gamma}||g||_{0,\gamma}
\end{align}
where the constant $C$ is independent of $\epsilon$. 
Combining (\ref{estimate1}) and (\ref{estimate2}), we conclude
\begin{align}\label{E^0}
    ||E_{0,\epsilon}^0||_{0,\gamma}\leq  C\epsilon^{2-2\gamma}||g||_{0,\gamma}.
\end{align}

\subsubsection{n=3}\label{n=3}
In a neighborhood $U_1$ of an infinity edge point, under local adapted coordinates, the operator $L$ takes the form 
\begin{align*}
L=x\partial_{xx}+xyb(x,y)\partial_{xy}+c(x,y)y^{2}\partial_{yy}+d(x,y)\partial_{x}+e(x,y)y\partial_{y}
\end{align*}
and the model operator:
\begin{align*}
L_{M}=x\partial_{xx}+y^{2}\partial_{yy}+d\partial_{x}+y\partial_{y}.
\end{align*}
We split 
\begin{gather}\label{error}
    L-L_M=L^r+e(x,y)y\partial_y,\\
    L^r=xyb(x,y)\partial_{xy}+\tilde{c}(x,y)y^2\partial_{yy}+\tilde{d}(x,y)\partial_x.
\end{gather}
There are two types of errors:
\begin{align*}
    E_{1,\epsilon}^{0,r}g=-\phi_\epsilon L^rA^t[\chi_\epsilon g],\ E_{1,\epsilon}^{0,t}g=-\phi_\epsilon e(x,y)y\partial_y A^t[\chi_\epsilon g].
\end{align*}
Then by Proposition (\ref{inhomogeneous}) and $[\chi_\epsilon]_\gamma=O(\epsilon^{-\gamma})$ and since
\begin{align*}
    |\chi_\epsilon(y)-\chi_\epsilon(y')|\leq C\epsilon^{-\gamma}|y-y'|^\gamma\leq C\epsilon^{-\gamma}|\text{ln} y-\text{ln}y'|^\gamma,
\end{align*}
we have
\begin{gather*}
    [E^{0,t}_{1,\epsilon}]_\gamma\leq C\epsilon^{-\gamma}T^{\frac{\gamma}{2}}||\chi_\epsilon g||_{0,\gamma}+CT^{\frac{\gamma}{2}}||\chi_\epsilon g||_{0,\gamma}\leq C\epsilon^{-2\gamma}T^{\frac{\gamma}{2}}||g||_{0,\gamma},\\
    ||E^{0,t}_{1,\epsilon}||_\infty\leq CT^{\frac{\gamma}{2}}||\chi_\epsilon g||_\infty\leq CT^{\frac{\gamma}{2}}\epsilon^{-\gamma}||g||_\infty.
\end{gather*}
Taken together
\begin{align}
    ||E^{0,t}_{1,\epsilon}||_{0,\gamma}\leq C\epsilon^{-2\gamma}T^{\frac{\gamma}{2}}||g||_{0,\gamma}.
\end{align}

Next we estimate the remaining term $E^{0,r}_{1,\epsilon}g$.
Since $\phi_\epsilon$ is supported on $x\leq 5\epsilon^2, y\leq 5\epsilon$, the vanishing properties of the coefficients of $L^r$ implies that the $L^\infty$-term is bounded by  
\begin{align*}
    ||E_{1,\epsilon}^{0,r}(g)||_\infty\leq C\epsilon||\chi_{\epsilon}g||_{0,\gamma}\leq C\epsilon^{1-\gamma}||g||_{0,\gamma}.
\end{align*}
To estimate the H\"older semi-norm we need to consider terms 
\begin{gather*}
    ||\phi_\epsilon xyb(x,y)||_\infty\left[\partial_{xy}A^t[\chi_\epsilon g]\right]_\gamma,
||\phi_\epsilon\tilde{c}(x,y)||_\infty y^2\partial_{yy}\left[A^t[\chi_\epsilon g]\right]_\gamma,\\ ||\phi_\epsilon\tilde{d}(x,y)||_\infty\partial_x\left[A^t[\chi_\epsilon g]\right]_\gamma
\end{gather*}
and 
\begin{gather*}
    \left[\phi_\epsilon xyb(x,y)\right]_\gamma||\partial_{xy}A^t[\chi_\epsilon g]||_\infty, \left[\phi_\epsilon\tilde{c}(x,y)\right]_\gamma||y^2\partial_{yy}A^t[\chi_\epsilon g]||_\infty,\\ \left[\phi_\epsilon\tilde{d}(x,y)\right]_\gamma||\partial_xA^t[\chi_\epsilon g]||_\infty.
\end{gather*}
The first three terms are bounded by 
\[C\epsilon||\chi_{\epsilon}g||_{0,\gamma}\leq C\epsilon^{1-\gamma}||g||_{0,\gamma}.\]
For any $0<\gamma'\leq\gamma<1$, the second three terms are bounded by 
\[C\epsilon^{1-\gamma}||\chi_\epsilon g||_{0,\gamma'}\leq C\epsilon^{1-\gamma-\gamma'}||g||_{0,\gamma}.\]
We therefore fix a $0<\gamma'\leq\gamma$ so that $\gamma+\gamma'<1$, then 
\begin{align}\label{Er}
    ||E^{0,r}_{1,\epsilon}g||_{0,\gamma}\leq C\epsilon^{1-\gamma-\gamma'}||g||_{0,\gamma}.
\end{align}

\subsubsection{n=1,2,4} 
These cases are estimated as above. We still split $L-L_M$ into two parts $L^r$ and the first tangential parts like (\ref{error}). The vanishing properties of coefficients of $L^r$ and the estimates of first tangential parts give the same results as
\begin{align*}
||E^{0,t}_{1,\epsilon}||_{0,\gamma}\leq C\epsilon^{-2\gamma}T^{\frac{\gamma}{2}}||g||_{0,\gamma},\\
    ||E^{0,r}_{1,\epsilon}g||_{0,\gamma}\leq C\epsilon^{1-\gamma-\gamma'}||g||_{0,\gamma}.
\end{align*}

\subsubsection{Proof of Theorem \ref{heat_equation}} 
\paragraph{Case when k=0.} We first prove the case when $k=0$. We write the solution $w=u+v$, where $v$ solves the homogeneous Cauchy problem with initial data $v(0,x,y)=f(x,y)$ and $u$ solves the inhomogeneous problem with $u(0,x,y)=0$. 

Recall the perturbation term (\ref{perturbation}), 
\begin{gather*}
    (\partial_{t}-L)\widehat{Q^t}g=g+E_{\epsilon}^0(g)+E_{\epsilon}^1(g)+E_i^\infty(g).
\end{gather*}
The support of the kernel of $E_i^\infty g$ has a positive distance from the diagonal and therefore this is a compact operator in the metric topology of $C^{0,\gamma}(P\times [0,T])$, 
tending to zero as $T\to0$. Using the estimations in the previous subsections,
\begin{gather*}
    ||E_{\epsilon}^0(g)+E_{\epsilon}^1(g)||_{0,\gamma, T}
    \leq C\left[\epsilon^{2-2\gamma}+\epsilon^{-2\gamma}T^{\frac{\gamma}{2}}+\epsilon^{1-\gamma-\gamma'}\right]\cdot||g||_{0,\gamma,T}.
\end{gather*}
Fix a $\delta>0$, we choose $\epsilon$ so that $\epsilon^{2-2\gamma}+\epsilon^{1-\gamma-\gamma'}=\delta$. If we choose $T_{0}$ sufficiently small, then the operator $E^{t}=E_{\epsilon}^{0,t}+E_{\epsilon}^{1,t}+E_{i}^\infty$ has norm strictly less than 1, and therefore the operator $Id+E^t$ is invertible as a map from $C^{k,\gamma}(P\times[0,T_{0}])$ to itself. Thus the operator 
\begin{align*}
    Q^{t}=\widehat{Q^{t}}\left(Id+E^{t}\right)^{-1}
\end{align*}
is the right inverse to $(\partial_{t}-L)$ up to time $T_{0}$ and is a bounded map
\begin{align*}
    Q^{t}:C^{0,\gamma}(P\times[0,T_{0}])\to C^{0,2+\gamma}(P\times[0,T_{0}]). 
\end{align*}

After constructing $\widehat{Q^t}$, the
solution operator for the inhomogeneous problem, we build a similar boundary parametrix for the
homogeneous Cauchy problem, which we then glue to the exact solution
operator for $P_U$. For each $n$ let $\widehat{Q_{n,i,j}^t}$ be the solution operator for the homogeneous Cauchy problem defined
by the model operator in $U_{i,j}$. And let $\widehat{Q^t_{00}}$ be the exact solution operator for the Cauchy problem $(\partial_t-L)u = 0$ on $W_0$ with Dirichlet data
on $bW_0\times [0,\infty)$.
For each $\epsilon>0$, we define a parametrix by setting 
\begin{align}
    \widehat{Q_0^{t}}=\sum_{i=0}^{4}\underset{i,j}{\sum}\phi_{n,i,j,\epsilon}\widehat{Q^t_{n,i,j}}\chi_{n,i,j,\epsilon}+\psi\widehat{Q^t_{00}}(1-\varphi_U).
\end{align}
Then
\begin{align*}
    (\partial_t-L)\widehat{Q^t_0}=E^tf: C^{0,2+\gamma}(P)\to  C^{0,\gamma}(P\times [0,T_0])
\end{align*}
is a bounded map and a slightly stronger statement is $\lim_{t\to 0}||E^t||=0$.

Finally we set $Q^t_0f=\widehat{Q^t_0}f-\widehat{Q}^tE^tf$, then 
\begin{align*}
    Q^t_0: C^{0,2+\gamma}(P)\to C^{0,2+\gamma}(P\times [0,T_0])
\end{align*}
is bounded and so is the solution operator of the Cauchy problem.


\paragraph{Higher order regularity.} We want to establish the convergence of $(Id+E^t)^{-1}$ in the operator norm defined by $C^{k,\gamma}(P\times[0,T_0])$. We follow the proof in \cite{EM} using the following lemma:
\begin{lemma}\label{banach}(\cite{EM}, Theorem 11.8.1)
Suppose that we have a ladder of Banach spaces 
\[X_0\supset X_1\supset X_2\supset  \cdot\cdot\cdot,\]
with norms $||\cdot||_k$, satisfying
$||x||_{k-1}\leq ||x||_k$ for all $x\in X_k$. Fix any $K\in\mathbb{N}$. Assume that $A$ is a linear map so that
$AX_k\subset X_k$ for every $k\in\mathbb{N}$, and that there are non-negative constants
$\{\alpha_j:j=0, 1,\cdot\cdot\cdot,K\}$ and $\{\beta_j:j=1,\cdot\cdot\cdot,K\}$, with
$\alpha_j<1$ for $0\leq j\leq K$,
for which we have the estimates:
\begin{gather*}
    ||Ax||_0\leq\alpha_0||x||_0\  \text{for}\ x\in X_0,\\
    ||A||_k\leq\alpha_k||x||_k+\beta_k||x||_{k-1}\ \text{for}\  x\in X_k.
\end{gather*}
In this case the Neumann series
\[(Id-A)^{-1}=\sum_{j=0}^\infty A^j\]
converges in the operator norm topology defined by $(X_k, ||\cdot||_k)$ for all $k\in
\{0,\cdot\cdot\cdot,K\}$.
\end{lemma}
It is convenient to induce the equivalent norms:
\begin{gather}
    ||g||_{k,\gamma,T}=||g||_{k-1,\gamma,T}+\underset{|\alpha|+|\beta|+2l=k}{sup}||\partial_t^l\partial_x^\alpha D_y^\beta g||_{0,\gamma,T},\\
    ||g||_{k,2+\gamma,T}=||g||_{k-1,2+\gamma,T}+\underset{|\alpha|+|\beta|+2l=k}{sup}||\partial_t^l\partial_x^\alpha D_y^\beta g||_{0,2+\gamma,T}.
\end{gather}
To apply Lemma \ref{banach}, it is sufficient to prove that there are $T_0>0$ and constants $\alpha_k<1, \beta_k$ such that
\begin{gather}
    ||E^tg||_{k,\gamma,T_0}\leq\alpha_k||g||_{k,\gamma,T_0}+\beta_k||g||_{k-1,\gamma,T_0}.
\end{gather}
The new terms in going from $k-1$ to $k$ are 
\[||\partial_t^l\partial_x^\alpha D_y^\beta E^tg||_{0,\gamma,T}\]
where $2l+|\alpha|+|\beta|=k$, which are of the terms of the types:
\begin{gather}
    \label{I}
    ||[\phi,L]A^{b+\alpha,t}[L_{b+\alpha}^m\partial_{x'}^\alpha\partial_{y'}^\beta]\chi_\epsilon g||_{0,\gamma,T},\  ||\phi e(x,y)D_yA^{b+\alpha,t}[L_{b+\alpha}^m\partial_{x'}^\alpha\partial_{y'}^\beta]\chi_\epsilon g||_{0,\gamma,T}\\
    \label{II}
    ||\partial_t^m\partial_x^\alpha D_y^\beta\left(\phi L^rA^{b,t}[\chi_\epsilon g]\right)||_{0,\gamma,T}
\end{gather}
based on the fact that
\begin{gather*}
    \partial_t^l\partial_x^\alpha D_y^\beta A_t^bg=A_t^{b+\alpha}[L_{b+\alpha}^l\partial_{x'}^\alpha D_{y'}^\beta g]+\sum_{q=0}^{l-1}L_{b+\alpha}^q\partial_t^{l-q-1}D_{y'}^\beta.
\end{gather*}
As analyzed in \ref{n=3}, the contribution of the terms in \eqref{I} to $||\cdot||_{k,\gamma,T}$ are bounded by 
\begin{gather*}
    C(T,k,\gamma)\epsilon^{-\mu(k,\gamma)}||g||_{k,\gamma,T},
\end{gather*}
for some $\mu(k,\gamma)>0$ and $C(T,k,\gamma)\to 0$ as $T\to 0$.
The contribution of the term \eqref{II} to $||\cdot||_{k,\gamma,T}$ is bounded by  
\begin{gather}
     C(k,\gamma)\epsilon^{1-\gamma-\gamma'}||g||_{k,\gamma,T}+C(T,k,\gamma)\epsilon^{-\mu(k,\gamma)}||g||_{k,\gamma,T}.
\end{gather}
Therefore we can choose suitable $\epsilon$ and sufficient small $T_0>0$ so that $\alpha_k<1$. This completes the proof.

We have proved the existence of solutions to the inhomogeneous problem and Cauchy problem in $[0,T_0]$, where $T_0$ does not depend on the initial data. To show the solution exists for all $t>0$, we can apply this proof again with initial data $f(0,\cdot)=v( T_0,\cdot),   g(t,\cdot)=g(t+T_0,\cdot)$. This extends the solution to $[0,2T_0]$. We can repeat this process $k$ times until $kT_0\geq t$.

\section{The Heat Kernel}\label{7} 
Now we are prepared to construct the heat kernel of the solution operator:
\begin{align}\label{kernel}
    (\partial_{t}-L)v_{f}=0\  \text{with}\  v_{f}(0,\cdot,\cdot)=f. 
\end{align}
We want to show the following result:
\begin{theorem}\label{heat kernel}
The global heat kernel $H_{t}(d_{1},d_{2},l_{1},l_{2})\in C^{\infty}(P^{reg}\times \mathring{P}\times(0,\infty))$ of the full operator $L$ exists and for $f\in C^0(P)$, then 
\begin{align*}
\int_P H_t(d_1,d_2,l_1,l_2)f(l_1,l_2)dl_1dl_2
\end{align*}
is the solution of $(\partial_{t}-L)v_f=0$ with $v_{f}(0,\cdot,\cdot)=f$. 
\end{theorem} 

When $(d_1, d_2)\in\mathring{P}$, the existence of the heat kernel follows from a standard construction for elliptic operators and the theory developed in the preceding sections. We may indeed construct a heat kernel $\tilde H_t(d_1,d_2,l_1,l_2)$ for $(l_1,l_2)$ in an open neighborhood of $(d_1,d_2)$ and obtain by standard elliptic regularity that $H_t-\tilde H_t$ solves \eqref{25} with a smooth right-hand side. A nontrivial aspect of the above result is that we also construct a heat kernel when $(d_1, d_2)$ is on the regular part of $bP$, where elliptic regularity is not applied. Instead we use fundamental solutions of model operators and use a series expansion to construct and analyze the heat kernel.


\subsection{Overview of main ideas}
We continue working under the similar $\epsilon-$grid covering constructed in Section \ref{6.1}. Instead of requiring $L$ having  normal forms in each $U_{n,i,j}$, we require $L$ takes the form in (\ref{definition}) but without mixed second-order derivative term. This can be achieved by taking $\epsilon$ small enough and through coordinate change.

To construct the global heat kernel of the full operator $L$, we first construct the local heat kernel $\{q_{n,i}^{t}\}$ in $U_{n,i}$, which satisfies that for $\forall q\in\mathring{U}_{n,i}$,
\begin{align}\label{31}
    (\partial_{t}-L)q_{n,i}^{t}(\cdot,q)=0,\quad  \lim_{t\to 0}\ q_{ij}^{t}(\cdot,q)=\delta_q(\cdot).
\end{align}
Then we patch them together to construct the global kernel parametrix.

\paragraph{Construction of Local heat kernel:}
Fix $q$ in the interior of $U_{n,i}$, we fix a smooth compactly supported function $h(x,y)\in C_{c}^{\infty}(U_{n,i})$ and $h(x,y)\equiv 1$ on Supp $\phi_{n,i}$. Letting 
\begin{align}\label{35}
    \widetilde{L}=L_{q,M}+h(x,y)(L-L_{q,M})
\end{align}
where $L_{q,M}$ is the model operator we choose corresponding to $q$.
Then $\widetilde{L}$ equals to $L$ on Supp $\phi_{i,j}$, so the problem transform to construct the Green function $q_{t}$ of the solution operator in this neighborhood:
\begin{align}\label{local L}
    (\partial_{t}-\widetilde{L})q_{t}(\cdot,q)=0\ \  \text{with}\ \  \underset{t\to0^{+}}\lim q_{t}(q,\cdot)=\delta(q).
\end{align}
The idea is to approximate $q_{t}(\cdot,q)$ by the kernel of the model operator $L_{M}$ modeled at $q$ and control the perturbation term to be sufficiently smooth. 

Let $K_t$ be the fundamental solution of $L_{q,M}$, $\textbf{K}_t$ be the solution operator of $L_{q,M}$, $\textbf{B}_t=(\widetilde{L}-L_{q,M})\textbf{K}_t$, an observation is that for $N\geq 1$,
\begin{align}\label{iteration1}
    (\partial_{t}-\widetilde{L})\left(q_{t}(\cdot,q)-K_{t}(t,\cdot,q)-\sum_{i=1}^{N-1}\textbf{K}_{\textbf{t}}^{i}\textbf{B}\delta\right) = \textbf{B}^{N}\delta
\end{align}
where $\textbf{B}\delta=(\widetilde{L}-L_M)K_t$.


\subsection{Existence of Local Heat Kernel} 
\subsubsection{Case {$q\in U_0$}.}
We first consider the case when $q$ is in a neighborhood $U_0$ of a regular corner. Under the local adapted coordinates (\ref{c_reg}), $L$ takes the form
\begin{align*}
L=a(x_1,x_2)x_1\partial_{x_1}^2+b(x_1,x_2)x_2\partial_{x_2}^2+d(x_1,x_2)\partial_{x_1}+e(x_1,x_2)\partial_{x_2}.
\end{align*}
Fix a point $q=(x_2,y_2)$ in the interior of $U_0$. We introduce the model operator $L_{q,M}$ in $S_{c\_reg}$:
\begin{align}\label{c_reg1bis}
    L_{q,M}=a(q)x_{1}\partial_{x_1}^2+b(q)x_{2}\partial_{x_{2}}^2+d(q)\partial_{x_{1}}+e(q)\partial_{x_{2}}.
\end{align}
Then the kernel formula $K_t(x,y,x_1,y_1)$ is the product of two one-dimensional kernel formula:
\begin{align}\label{ppp}
    K_t(x,y,x_2,y_2)=p^{d(q)/a(q)}_{a(q)t}(x,x_2)p^{e(q)/b(q)}_{b(q)t}(y,y_2).
\end{align}

\begin{proposition}\label{iteration}
Denote  $d(t,x,y)=\frac{(\sqrt{x}-\sqrt{x_2})^2}{2a(q)t}+\frac{(\sqrt{y}-\sqrt{y_2})^2}{2b(q)t}$.
Assume that for $(x,y)\in\mathbb{R}_+^2$, $j\leq\frac{3}{2}$
\begin{align*}
     |g(t,x,y)|\leq\frac{1}{t^j}e^{-d(t,x,y)}
\end{align*}
and for some $0<\gamma<1$,
\begin{align}\label{g_holder}
    |g(t,x,y)-g(t,x',y)|\leq\frac{1}{t^{j+\frac{\gamma}{2}}}|\sqrt{x}-\sqrt{x'}|^{\gamma}\left(e^{-d(t,x,y)}+e^{-d(t,x',y)}\right)\\
    |g(t,x,y)-g(t,x,y')|\leq\frac{1}{t^{j+\frac{\gamma}{2}}}|\sqrt{y}-\sqrt{y'}|^{\gamma}\left(e^{-d(t,x,y)}+e^{-d(t,x',y)}\right).
\end{align}
Then there exists a constant $C>0$ such that
\begin{align*}
     |\textbf{B}g(t,x,y)|\leq C\frac{\sqrt{t}}{t^j}e^{-\frac{d(t,x,y)}{2}}
\end{align*}
and 
\begin{align*}
    |\textbf{B}g(t,x,y)-\textbf{B}g(t,x',y)|\leq\frac{C}{t^{j+\frac{\gamma-1}{2}}}|\sqrt{x}-\sqrt{x'}|^{\gamma}\left(e^{-\frac{d(t,x,y)}{2}}+e^{-\frac{d(t,x',y)}{2}}\right)\\
    |\textbf{B}g(t,x,y)-\textbf{B}g(t,x,y')|\leq\frac{C}{t^{j+\frac{\gamma-1}{2}}}|\sqrt{y}-\sqrt{y'}|^{\gamma}\left(e^{-\frac{d(t,x,y)}{2}}+e^{-\frac{d(t,x',y)}{2}}\right).
\end{align*}
\end{proposition}
\begin{proof}
By replacing $a(q)t, b(q)t$ with $t$ in the kernel formula
(\ref{ppp}), we might as well assume that $a(q)=b(q)=1$. Through a Taylor expansion of the coefficient of $L$ at $q$, we see
\begin{align}\label{taylor}
    L-L_{q,M}=\Theta((x-x_2)+(y-y_2))(L_x+L_y)
\end{align}
where $L_x=x\partial_x^2+d(q)\partial_x, L_y=y\partial_y^2+e(q)\partial_y$. In the following proof we will use a number of kernel estimates, which we present and prove in the Appendix. 

First we assume that $d(q)\geq\frac{1}{2}$ or $d(q)=0$, $e(q)\geq\frac{1}{2}$ or $e(q)=0$. We use the estimate of $p^d_t(x,y)$ (\ref{estimate_p}) to see
\[p^d_{t-s}(x,x_1)\leq\frac{C}{\sqrt{t-s}}e^{-\frac{(\sqrt{x}-\sqrt{x_1})^2}{2(t-s)}}\frac{1}{\sqrt{x_1}},\]
so we deduce that:
\begin{align*}
    |\textbf{K}_tg|\leq Ce^{-\frac{(\sqrt{x}-\sqrt
    x_2)^2+(\sqrt{y}-\sqrt{y_2})^2}{2t}}.
\end{align*}

Next we turn to estimate $\textbf{B}g$. It suffices to consider $L_x$ so that
\begin{align}
    \widetilde{L}-L_M=\Theta((x-x_2)+(y-y_2))L_x.
\end{align}
Since the $t-$regularity of $(\widetilde{L}-L_M)K_{t-s}$ is $(t-s)^{-2}$ which is not integrable in $s$, using the fact that
\[\int_0^\infty x\partial_x^2p^d_t(x,y)dy=0,\]
we split $\textbf{B}g$ into two parts
\begin{gather}\label{Bg}
    \textbf{B}g=(\tilde{L}-L_M)\int_{\frac{t}{2}}^t\int_{\mathbb{R}_+^2} K_{t-s}(x,y,x_1,y_1)(g(s,x_1,y_1)-g(s,x,y_1))dx_1dy_1ds\\
    +(\tilde{L}-L_M)\int_0^{\frac{t}{2}}\int_{\mathbb{R}_+^2} K_{t-s}(x,y,x_1,y_1)g(s,x_1,y_1)dx_1dy_1ds.
\end{gather}
The first term is bounded by 
\[(\tilde{L}-L_M)\int_{\frac{t}{2}}^t\int_{\mathbb{R}_+^2} K_{t-s}(x,y,x_1,y_1)\frac{1}{s^{j+\frac{\gamma}{2}}}|\sqrt{x}-\sqrt{x_1}|^\gamma\left(e^{-\frac{(\sqrt{x_1}-\sqrt{x_2})^2}{2s}}+e^{-\frac{(\sqrt{x}-\sqrt{x_2})^2}{2s}}\right) e^{-\frac{(\sqrt{y_1}-\sqrt{y_2})^2}{2s}}dx_1dy_1ds.\]
We use the estimates of derivatives (\ref{second_estimate}), (\ref{first_estimate}) to see 
\begin{align}\label{L_x}
    |\partial_xp^d_{t-s}(x,x_1)|, |x\partial_x^2p^d_{t-s}(x,x_1)|\leq\frac{C}{(t-s)^\frac{3}{2}}e^{-\frac{(\sqrt{x}-\sqrt{x_1})^2+(y^s-y^s_1)^2}{2(t-s)}}\frac{1}{\sqrt{x_1}}.
\end{align}
By the integral
\begin{align}\label{integral}
    \int_{-\infty}^\infty\frac{1}{\sqrt{t-s}}e^{-\frac{(x-x_1)^2}{t-s}}\frac{1}{\sqrt{s}}e^{-\frac{(x_1-x_2)^2}{s}}dx_1= \frac{1}{\sqrt{t}}e^{-\frac{(x-x_2)^2}{t}}
\end{align}
and the contribution of coefficients of $\tilde{L}-L_M$:
\begin{align}\label{t_sqrt}
   |x-x'|\cdot e^{-\frac{(x-x')^2}{t}}\leq C\sqrt{t}e^{-\frac{(x-x')^2}{2t}}, 
\end{align}
we can deduce that (\ref{Bg}) is integrable and
\[|\textbf{B}g|\leq C\frac{1}{\sqrt{t}}e^{-\frac{(\sqrt{x}-\sqrt{x_2})^2+(\sqrt{y}-\sqrt{y_2})^2}{2t}}.\]

Finally we turn to estimate $[\cdot]_\gamma$ part of $\textbf{B}g$. We consider two cases.
In the first case when $|\sqrt{x}-\sqrt{x'}|\geq\sqrt{t}$,
\[|\textbf{B}g(t,x,y)-\textbf{B}g(t,x',y)|\leq |\textbf{B}g(t,x,y)+\textbf{B}g(t,x',y)|\cdot\left(\frac{|\sqrt{x}-\sqrt{x'}|}{t}\right)^\gamma.\]
In the second case when $|\sqrt{x}-\sqrt{x'}|\leq \sqrt{t}$,  
\[|\sqrt{x}-\sqrt{x_2}|\leq |\sqrt{x'}-\sqrt{x_2}|+\sqrt{t},\]
so that multiplying this term with $(e^{-\frac{(\sqrt{x}-\sqrt{x_2})^2}{t}}+e^{-\frac{(\sqrt{x'}-\sqrt{x_2})^2}{t}})$,
\begin{align}\label{x-x'}
    |\sqrt{x}-\sqrt{x_2}|\cdot(e^{-\frac{(\sqrt{x}-\sqrt{x_2})^2}{t}}+e^{-\frac{(\sqrt{x'}-\sqrt{x_2})^2}{t}})\leq C\sqrt{t}(e^{-\frac{(\sqrt{x}-\sqrt{x_2})^2}{2t}}+e^{-\frac{(\sqrt{x'}-\sqrt{x_2})^2}{2t}}).
\end{align}

\begin{enumerate}
    \item If $x'<\frac{x}{3}$, it suffices to show that 
    \begin{align*}
        |x\partial_x^2\textbf{K}_tg(t,x,y)|\leq C\frac{x^{\frac{\gamma}{2}}}{t^{j+\frac{\gamma}{2}}}e^{-\frac{d(x,y)}{2t}}.
    \end{align*}
    This is obtained similarly to $\textbf{B}g$ in (\ref{Bg}): we replace the estimate (\ref{L_x}) with
    \begin{align*}
        x\partial_x^2p(t-s,x,x_1)dx_1\leq\frac{Cx^{\frac{\gamma}{2}}}{(t-s)^\frac{3+\gamma}{2}}e^{-\frac{(\sqrt{x}-\sqrt{x_1})^2+(y^s-y^s_1)^2}{2(t-s)}}dx^s_1.
    \end{align*}
    Then
    \begin{gather*}
        |(x-x_2)x\partial_x^2\textbf{K}_tg(t,x,y)-(x'-x_2)x'\partial_x^2\textbf{K}_tg(t,x',y)|\leq
        |(x-x_2)-(x'-x_2)|x\partial_x^2\textbf{K}_tg(t,x,y)\\
        +|(x'-x_2)|\cdot|x\partial_x^2\textbf{K}_tg(t,x,y)-x'\partial_x^2\textbf{K}_tg(t,x',y)|.
    \end{gather*}
    The first term is bounded by $C\frac{|\sqrt{x}-\sqrt{x'}|^\gamma}{t^{j+\frac{\gamma-1}{2}}}e^{-\frac{d(x,y)}{2t}}$. The estimate below
    \begin{gather*}
        |x\partial_x^2\textbf{K}_tg(t,x,y)-x'\partial_x^2\textbf{K}_tg(t,x',y)|\leq C\frac{x^{\frac{\gamma}{2}}}{t^{j+\frac{\gamma}{2}}}\left(e^{-\frac{d(x,y)}{2t}}+e^{-\frac{d(x',y)}{2t}}\right)\\
        \leq C\frac{|\sqrt{x}-\sqrt{x'}|^\gamma}{t^{j+\frac{\gamma}{2}}}\left(e^{-\frac{d(x,y)}{2t}}+e^{-\frac{d(x',y)}{2t}}\right)
    \end{gather*}
    and (\ref{x-x'}) show that 
    the second term is bounded by $C\frac{|\sqrt{x}-\sqrt{x'}|^\gamma}{t^{j+\frac{\gamma-1}{2}}}e^{-\frac{d(x,y)}{2t}}$.

    \item We assume that $\frac{x}{3}<x'<x$,
    
    Since we have established the H\"older continuity of the first derivative, it suffices to establish the H\"older continuity when $\tilde{L}-L_M=(x-x_2)(x\partial^2_x+d\partial_x)$. We split $\textbf{B}^2g$ into two parts
\begin{gather}\label{B^2g}
    \textbf{B}g(t,x,y)-\textbf{B}g(t,x',y)  =\\
    (\tilde{L}-L_M)\int_{\frac{t}{2}}^t\int_{\mathbb{R}^2_+}\left[ K_{t-s}(x,y,x_1,y_1)-K_{t-s}(x',y,x_1,y_1)\right]g(s,x_1,y_1)dx_1dy_1ds\\
    +(\tilde{L}-L_M)\int_0^{\frac{t}{2}}\int_{\mathbb{R}^2_+}  \left[K_{t-s}(x,y,x_1,y_1)-K_{t-s}(x',y,x_1,y_1)\right]g(s,x_1,y_1)dx_1dy_1ds\\
    =A+B.
\end{gather}
With $J=[\alpha,\beta]$, $\sqrt{\alpha}=\frac{3\sqrt{x'}-\sqrt{x}}{2},\ \sqrt{\beta}=\frac{3\sqrt{x}-\sqrt{x'}}{2}$, we have
\begin{gather}
    A=\int_{\frac{t}{2}}^t[\int_{J}(\tilde{L}-L_M)K_{t-s}(x,x_1,y,y_1)(g(s,x_1,y_1)-g(s,x,y_1))dx_1dy_1-\\
    \int_{J}(\tilde{L}-L_M)K_{t-s}(x',x_1,y,y_1)(g(s,x_1,y_1)-g(s,x',y_1))dx_1dy_1-\\
    \int_{J^c}(\tilde{L}-L_M)K_{t-s}(x,x_1,y,y_1)(g(s,x,y_1)-g(s,x',y_1))dx_1dy_1+\\
    \int_{J^c}((\tilde{L}-L_M)K_{t-s}(x,x_1,y,y_1)-(\tilde{L}-L_M)K_{t-s}(x',x_1,y,y_1))(g(s,x_1,y_1)-g(s,x',y_1))dx_1dy_1]ds\\
    =I_1+I_2+I_3+I_4.
\end{gather}
\end{enumerate}
We first estimate $I_3$. Based on the observation that for $t>0$,
\[\partial_tp_d(t,x,x_1)=L_{d,x}p_d(t,x,x_1)=L^*_{d,x_1}p_d(t,x,x_1),\]
the operator $L^t_x=\partial_{x_1}(\partial_{x_1}x_1-d)$, so we can perform $x_1$-integral to obtain that 
\begin{align*}
    I_3\leq\int_{\frac{t}{2}}^t\left[(\partial_{x_1}x_1-d)K_{t-s}(x,\alpha,y,y_1)-(\partial_{x_1}x_1-d)K_{t-s}(x,\beta,y,y_1)\right](g(s,x,y_1)-g(s,x',y_1))dx_1dy_1ds.
\end{align*}
We use the H\"older estimate of $g$ (\ref{g_holder})
and \cite{EM} Lemma 7.1.16 to complete this step. The needed bounds for first two terms $I_1, I_2$ are given in \cite{EM} Lemma 7.1.17.
This leaves only the terms $B$ and $I_4$. We use that
\[|LK_{t-s}(x,x_1,y,y_1)-LK_{t-s}(x',x_1,y,y_1)|\leq\frac{C|\sqrt{x}-\sqrt{x'}|^\gamma }{(t-s)^{\frac{3+\gamma}{2}}} \left(e^{-\frac{(\sqrt{x}-\sqrt{x_1})^2}{2(t-s)}}+e^{-\frac{(\sqrt{x'}-\sqrt{x_1})^2}{2(t-s)}}\right)e^{-\frac{(y^s-y^s_1)^2}{2(t-s)}}\]
and again the H\"older estimate of $g$ in (\ref{g_holder}).
On $J^c$, $\frac{1}{3}\leq\frac{|x'-x_1|}{|x-x_1|}\leq 3$ is bounded below and above, then (\ref{t_sqrt}) combined with (\ref{x-x'}) show that $I_4$ is integrable. $B$ is estimated similarly. Taken together, we obtain the result.

If $0<d(q)\leq\frac{1}{2}$ or $0<e(q)\leq\frac{1}{2}$, all the arguments are essentially similar to those above except that we need to replace the integral (\ref{integral}) with the estimate: for $0<d<1$,
\begin{align}
    \int_{-\infty}^\infty\frac{1}{(t-s)^d}e^{-\frac{(\sqrt{x}-\sqrt{x_1})^2}{t-s}}\frac{1}{x_1^{1-d}}\frac{1}{\sqrt{s}}e^{-\frac{(\sqrt{x_1}-\sqrt{x_2})^2}{s}}dx_1\leq C_p\frac{1}{(t-s)^d\sqrt{s}}\left(\sqrt{\frac{(t-s)s}{t}}\right)^{2d}e^{-\frac{(\sqrt{x}-\sqrt{x_2})^2}{t}}.
\end{align}
Notice that the degree of time $t$ on the right hand side above is independent of $d$, so in total the regularity of time $t$ in this case is same as that when $d>\frac{1}{2}$.
\end{proof}
Then we can show the existence of local heat kernel and analyze their behavior as $t\to 0$.
\begin{proposition}
There exists a local heat kernel satisfying \eqref{31} in $U_0$. Moreover, for any $f\in C_c(U_0)$, 
\begin{align}\label{kernel0}
(\partial_t-L)\int_{U_0} q_t(x,y,x_1,y_1)f(x_1,y_1)dx_1dy_1=0,
\end{align}
and 
\begin{enumerate}\label{delta}
    \item If $(0,0)$ is a transverse point,  \begin{align}\label{kernel1}
    \lim_{t\to 0}\int_{U_0} q_t(0,0,x_1,y_1)f(x_1,y_1)dx_1dy_1=f(0,0).
\end{align}
\item If $(0,0)$ is a tangent point,
\begin{align}\label{kernel2}
    \lim_{t\to 0}\int_{U_0} q_t(0,0,x_1,y_1)f(x_1,y_1)dx_1dy_1=0.
\end{align}
\end{enumerate}
\end{proposition}
\begin{proof}
First we show that $\textbf{B}\delta=(\tilde{L}-L_M)K_t$ satisfies the assumption in Proposition \ref{iteration}. We begin by writing
\begin{align*}
    K_t(x,y,x_1,y_1)=p^{d(x_1,y_1)}_t(x,x_1)p_t^{e(x_1,y_1)}(y,y_1).
\end{align*}
\paragraph{1. If $d(0,0)>0$:}by shrinking the neighborhood $U_0$, we might as well assume that $d(x,y)\in [d_1,d_2]$ for some $0<d_1<d_2$. Then by the estimates in \eqref{estimate_p}, \eqref{second_estimate}, \eqref{first_estimate},
\begin{gather*}
    |\partial_xp_{d(x_1,y_1)}(t,x,x_1)|, |x\partial_x^2p_{d(x_1,y_1)}(t,x,x_1)|\leq\frac{C}{t^{\frac{3}{2}}}e^{-\frac{(\sqrt{x}-\sqrt{x_1})^2}{2t}}(x_1^{d_1-1}+\frac{1}{\sqrt{x_1}}),\\
    |p_{d(x_1,y_1)}(t,x,x_1)|\leq\frac{C}{\sqrt{t}}e^{-\frac{(\sqrt{x}-\sqrt{x_1})^2}{2t}}(x_1^{d_1-1}+\frac{1}{\sqrt{x_1}}).
\end{gather*}
\paragraph{2. If $d(0,0)=0$:}so $d(0,y)=0$ by Assumption \ref{assumption}. Again by shrinking $U_0$ we may assume that $d(x,y)\in [0,\frac{1}{2}]$.
A more precise estimate is
\begin{align*}
    p_{d(x,y)}(t,x,y)
    \leq\begin{cases}
\begin{array}{cc}
\frac{1}{t^{d(x,y)}}e^{-\frac{x+y}{t}}y^{d(x,y)-1}\psi_{d(x,y)}(\frac{xy}{t^2}) & \text{when}
    \ \frac{xy}{t^2}<1\\
\frac{1}{\sqrt{yt}}e^{-\frac{(\sqrt{x}-\sqrt{y})^2}{t}} & \text{\ when}
    \ \frac{xy}{t^2}>1
\end{array}.\end{cases}
\end{align*}
Using the expansion of $\psi_d$ at $0$, for $0\leq z\leq 1$,
\[\psi_d(z)=\psi_d(0)+\psi_d'(c)z\]
for some $c\in [0,1]$.  $\psi_d'$ is uniform bounded on $d\in (0,\frac{1}{2}]\times[0,1]$. 
Since $\psi_d(0)=\frac{1}{\Gamma(d)}$, $\frac{1}{\Gamma(0)}=0$ and $\frac{\partial}{\partial d}\frac{1}{\Gamma(d)}$ is uniformly bounded on $(0,\frac{1}{2}]$, where $\Gamma(\cdot)$ is the usual Gamma function, we have
\[\psi_{d(x,y)}(\frac{xy}{t^2})\leq\frac{1}{\Gamma(d(x,y))}+C\frac{xy}{t^2}\leq C(y+\frac{xy}{t^2}).\]
In all 
\begin{align*}
    p_{d(x,y)}(t,x,x_1)\leq C\frac{1}{\sqrt{x_1t}}e^{-\frac{(\sqrt{x}-\sqrt{x_1})^2}{2t}}.
\end{align*}
Similarly we can show that the derivatives of $p_{d(x,y)}(x,x_1)$ are bounded by
\begin{align*}
    |\partial_xp_{d(x_1,y_1)}(t,x,x_1)|, |x\partial_x^2p_{d(x_1,y_1)}(t,x,x_1)|\leq\frac{C}{t^{\frac{3}{2}}}e^{-\frac{(\sqrt{x}-\sqrt{x_1})^2}{2t}}\frac{1}{\sqrt{x_1}}.
\end{align*}

Above all, there exists $d_1,d_2>0$ such that 
\begin{gather}
    |\partial_xp_{d(x_1,y_1)}(t,x,x_1)|, |x\partial_x^2p_{d(x_1,y_1)}(t,x,x_1)|\leq\frac{C}{t^{\frac{3}{2}}}e^{-\frac{(\sqrt{x}-\sqrt{x_1})^2}{2t}}x_1^{d_1-1}\\
    |p_{d(x_1,y_1)}(t,x,x_1)|\leq\frac{C}{\sqrt{t}}e^{-\frac{(\sqrt{x}-\sqrt{x_1})^2}{2t}}x_1^{d_1-1},
\end{gather}
and 
\begin{gather}
    |\partial_yp_{e(x_1,y_1)}(t,y,y_1)|, |y\partial_y^2p_{e(x_1,y_1)}(t,y,y_1)|\leq\frac{C}{t^{\frac{3}{2}}}e^{-\frac{(\sqrt{y}-\sqrt{y_1})^2}{2t}}y_1^{d_2-1}\\
    |p_{e(x_1,y_1)}(t,y,y_1)|\leq\frac{C}{\sqrt{t}}e^{-\frac{(\sqrt{y}-\sqrt{y_1})^2}{2t}}y_1^{d_2-1}.
\end{gather}
Finally because 
\begin{align*}
    L-L_M=\Theta((x-x_2)+(y-y_2))(x\partial_x^2+\partial_x+y\partial_y^2+\partial_y),
\end{align*}
combined with the estimates above, we conclude that 
\begin{align}\label{bound}
    |(L-L_M)K_t|\leq\frac{C}{t^{\frac{3}{2}}}e^{-\frac{(\sqrt{x}-\sqrt{x_1})^2+(\sqrt{y}-\sqrt{y_1})^2}{2t}}x_1^{d_1-1}y_1^{d_2-1},
\end{align}
so we can iterate Proposition \ref{iteration} until $\textbf{B}^5\delta\in\mathcal{C}^{0,\gamma}(U_0\times [0,T])$, for which by the results in  Theorem \ref{heat_equation}, there exists the solution $Q^t\textbf{B}^5\delta\in\mathcal{C}^{0,2+\gamma}(P\times [0,T])$.
Recall that
\begin{align}\label{iteration2}
    (\partial_{t}-\widetilde{L})\left(q_{t}(\cdot,q)-K_{t}(t,\cdot,q)-\sum_{i=1}^{N-1}\textbf{K}_{\textbf{t}}\textbf{B}^i\delta\right) = \textbf{B}^{N}\delta,
\end{align}
hence there exists a local heat kernel $q_t(x,y,x_2,y_2)$ in $U_0$, which has the expansion
\begin{gather}\label{expansion}
    q_t(x,y,x_2,y_2)=K_t(x,y,x_2,y_2)+\sum_{i=1}^{4}\textbf{K}_{\textbf{t}}\textbf{B}^i\delta+Q^t\textbf{B}^5\delta.
\end{gather}

Moreover, the  derivatives of it \[\partial_xq_t(x,y,x_1,y_1), \partial_yq_t(x,y,x_1,y_1), x\partial_x^2q_t(x,y,x_1,y_1), y\partial_y^2q_t(x,y,x_1,y_1)\]
are also bounded by the right hand side of (\ref{bound}).

Therefore, on $[t_0,t_1]$, the integral of these derivatives converges uniformly on $t\in [t_0,t_1],(x,y)\in U_0$. Hence we can exchange derivatives with integration to obtain that
\begin{align}
(\partial_t-L)\int_{U_0} q_t(x,y,x_1,y_1)f(x_1,y_1)dx_1dy_1=\int_{U_0} (\partial_t-L)q_t(x,y,x_1,y_1)f(x_1,y_1)dx_1dy_1=0.
\end{align}

Finally we turn to prove (\ref{delta}). Based on the expansion (\ref{expansion}), it suffices to show that (\ref{delta}) holds for the fundamental solution 
\begin{align}\label{FUNDAMENTAL}
    K_t(0,0,x,y)=p^{d(x,y)}_t(0,x)\cdot p^{e(x,y)}_t(0,y).
\end{align}
\paragraph{1. If $(0,0)$ is a transverse point,} we assume that $d(x,y),e(x,y)\in [d_1,d_2]$ for some $0<d_1<d_2$.
We denote by
\begin{align}\label{kernel_at_0}
    K'_t(0,0,x,y)=p^{d(0,0)}_t(0,x)\cdot p^{e(0,0)}_t(0,y).
\end{align}
Since $\underset{t\to 0}{\lim}K_t'(0,0,x,y)=\delta(0,0)$, it is suffices to show that
\begin{align}\label{00}
    \lim_{t\to 0} (K_t(0,0,x,y)-K'_t(0,0,x,y))=0.
\end{align}
We have
\begin{gather*}
    |K_t(0,0,x,y)-K'_t(0,0,x,y)|\\
    \leq |(p^{d(x,y)}_t(0,x)-p^{d(0,0)}_t(0,x)||p^{e(x,y)}_t(0,y)|+|(p^{e(x,y)}_t(0,y)-p^{e(0,0)}_t(0,y)||p^{d(0,0)}_t(0,x)|\\
    =:I+II.
\end{gather*}
By the mean value theorem, there exists $c\in [d,e]$ such that
\begin{align*}
    |p^d(t,0,y)-p^e(t,0,y)|=|\frac{\partial}{\partial_d}p_c(t,0,y)|\cdot |d-e|.
\end{align*}
Since
\begin{gather*}
    |\frac{\partial}{\partial_d}p_c(t,0,y)|=\frac{1}{y}e^{-\frac{y}{t}}\left(\frac{y}{t}\right)^c|\frac{1}{\Gamma(c)}\text{ln}\frac{y}{t}-\frac{\Gamma'(c)}{\Gamma^2(c)}|\leq \frac{C}{y}e^{-\frac{y}{t}}\left(\frac{y}{t}\right)^c\left[ln\frac{y}{t}+1\right]\\
    \leq \frac{C}{y}e^{-\frac{y}{t}}(\left(\frac{y}{t}\right)^{d_1}+\left(\frac{y}{t}\right)^{d_2})\left[ln\frac{y}{t}+1\right],
\end{gather*}
hence 
\begin{align*}
    |p_{d_2(x,y)}(t,0,y)-p_{d_2(0,0)}(t,0,y)|\leq C(x+y)\frac{1}{y}e^{-\frac{y}{t}}(\left(\frac{y}{t}\right)^{d_1}+\left(\frac{y}{t}\right)^{d_2})\left[ln\frac{y}{t}+1\right]
\end{align*}by the Taylor expansion of $d_2$ at (0,0).
It is not hard to show that 
\begin{gather*}
     \lim_{t\to 0}\int_0^\infty xp_{d_1(0,0)}(t,0,x)dx=0\\
    \lim_{t\to 0}\int_0^\infty y\frac{1}{y}e^{-\frac{y}{t}}(\left(\frac{y}{t}\right)^M+\left(\frac{y}{t}\right)^N)\left[ln\frac{y}{t}+1\right]dy=0.
\end{gather*}
Therefore we have 
\[\lim_{t\to 0}\int_{U_0} II\cdot f(x_1,y_1)dx_1dy_1=0.\]
Also we can obtain that 
\[\lim_{t\to 0}\int_{U_0} I\cdot f(x_1,y_1)dx_1dy_1=0,\]
which leads to (\ref{00}).

\paragraph{2. If $(0,0)$ is a tangent point,} we assume that $e(0,0)=0$. Then $e(x,0)=0$ by the assumption (\ref{assumption}). We also assume that $e(x,y)\in [0,M]$. We have
\begin{gather*}
    p^{e(x,y)}_t(0,y)=\frac{1}{t^{e(x,y)-1}}y^{e(x,y)}e^{-\frac{y}{t}}\frac{1}{\Gamma(e(x,y))}\leq C\frac{1}{t^{e(x,y)-1}}y^{e(x,y)+1}e^{-\frac{y}{t}}\\
    \leq Ct(p^1_t(0,y)+p^{M+1}_t(0,y)).
\end{gather*}
If $d(0,0)=0$, then it is also bounded by $Ct(p^1_t(0,y)+p^{M+1}_t(0,y))$, so
\begin{align*}
    \lim_{t\to 0}\ K_t(0,0,x,y)=0
\end{align*}
in the sense of distribution.
If $d(0,0)>0$, by the estimation in the previous case,
\begin{align*}
    \lim_{t\to 0}\ (p_t^{d(x,y)}(0,x)-p_t^{d(0,0)}(0,x))=0,
\end{align*}
combined with 
\begin{align*}
    \lim_{t\to 0}\ p_t^{d(0,0)}(0,x)p^{e(x,y)}_t(0,y)=0,
\end{align*}
we obtain that
\begin{align}\label{000}
    \lim_{t\to 0}K_t(0,0,x,y)=0
\end{align}
in the sense of distribution.
\end{proof}

\subsubsection{Other cases}
\paragraph{1. $q$ in a neighborhood of a regular edge point $p$.}
When $q=(x_2,y_2)\in U_3$ is in a neighborhood of a regular edge point, under local adapted coordinates, $L$ takes the form
\begin{align*}
L=a(x,y)x\partial_x^2+b(x,y)\partial_{yy}+d(x,y)\partial_x+e(x,y)\partial_y.
\end{align*}
Replacing $d(t,x,y)$ in Proposition \ref{iteration} with
\begin{gather*}
    d(t,x,y)=\frac{(\sqrt{x}-\sqrt{x_2})^2}{2a(q)t}+\frac{(y-y_2)^2}{4b(q)t},
\end{gather*} 
we can similarly prove the existence of the local heat kernel through its series expansion 
\begin{gather*}
    q_t(x,y,x_2,y_2)=K_t(x,y,x_2,y_2)+\sum_{i=1}^{4}\textbf{K}_{\textbf{t}}\textbf{B}^i\delta+Q^t\textbf{B}^5\delta
\end{gather*}
where $K_t(x,y,x_2,y_2)=p^{d(q)/a(q)}_{a(q)t}(x,x_2)k^e_{b(q)t}(y,y_2)$. By comparing $K_t(p,q)$ with fundamental solution of model operator with coefficients frozen at $p$, we can derive the result like in \ref{00}, \ref{000}. That being said, if $p$ is a transverse point, then $\underset{t\to 0}{\lim}\  q_t(p,\cdot)=\delta(p)$, if $p$ is a tangent point, then $\underset{t\to 0}{\lim}\  q_t(p,\cdot)=0$. 

\paragraph{2. $q$ in a neighborhood of a mixed corner/infinity edge point $p$.} When $q=(x_2,y_2)\in U_1$ is in a neighborhood of a regular edge point, under local adapted coordinates \ref{cmix}, $L$ takes the form
\begin{align*}
L=a(x,y)x\partial_x^2+b(x,y)y^2\partial_{yy}+d(x,y)\partial_x+e(x,y)y\partial_y.
\end{align*}
We can similarly prove the existence of the local heat kernel through its series expansion 
\begin{gather*}
    q_t(x,y,x_2,y_2)=K_t(x,y,x_2,y_2)+\sum_{i=1}^{4}\textbf{K}_{\textbf{t}}\textbf{B}^i\delta+Q^t\textbf{B}^5\delta.
\end{gather*}
where $K_t(x,y,x_2,y_2)=p^{d(q)/a(q)}_{a(q)t}(x,x_2)\frac{1}{\sqrt{4\pi t}}exp\left[-\frac{(\text{ln} y-\text{ln}y_2)^2}{4b(x_2,y_2)t}\right]\frac{1}{y_2}$.
Particularly when $y=0$,
\[q_t(x,0,x_2,y_2)=p_t(x,x_2)\delta_0(y_2),\]
where $p_t(x,x_2)$ is a 1-dimensional heat kernel of Kimura operator on $y=0$.
This implies that the infinity edge is isolated from $\mathring{P}$, the diffusion starting from the $E_\infty$ stay on it. 

\paragraph{3. when $q$ in a neighborhood of an infinity corner $p$.}
When $q=(x_2,y_2)\in U_3$ is in a neighborhood of a regular edge point, under local adapted coordinates, $L$ takes the form
\begin{align*}
L=a(x,y)x^2\partial_x^2+b(x,y)y^2\partial_y^2+d(x,y)x\partial_{x}+e(x,y)y\partial_y.
\end{align*}
Again we can similarly prove the existence of the local heat kernel through its series expansion. 
$q_t(x,y,x_2,y_2)$ takes the form
\[p_t(x,y,x_2,y_2)exp\left[-\frac{(\text{ln}x-\text{ln}x_2)^2}{4a(x_2,y_2)t}-\frac{(\text{ln}y-\text{ln}y_2)^2}{4b(x_2,y_2)t}\right]\frac{1}{x_2y_2},\] 
where $p_t(x,y,x_2,y_2)$ is bounded for $t>0$. When $x=0$ or $y=0$,
\begin{gather*}
    q_t(0,y,x_2,y_2)=\delta_0(x_2)\cdot p_t(y,y_2)\\
    q_t(x,0,x_2,y_2)=p_t(x,x_2)\cdot \delta_0(y_2)
\end{gather*}
where $p_t$ is the 1-dimensional heat kernel on $y=0$. In particular
\begin{gather*}
    q_t(0,0,x_2,y_2)=\delta_{(0,0)}(x_2,y_2),
\end{gather*}
which means $(0,0)$ is an isolated point.

In all for $q\in\mathring{P}$, if $p\in P^{reg}$, the heat kernel $q_t(p,q)$ is well defined and continuous at $p$. In other cases, $q_t(p,q)=0$. Specifically $q_t(p, \cdot)$ degenerates to a 1-dimensional along the quadratic edge when $p\in E_\infty\cup C_{mix}$, and $q_t(p, \cdot)$ is the delta function at $p$ when $p\in C_\infty$.  On the other hand for $p\in P^{reg}$, we investigated the limit behavior of $q_t(p,\cdot)$ when $t\to 0$. It tends to $\delta(p)$ if $p$ is a transverse point and tends to 0 in the sense distribution if $p$ is a tangent point. We summarize this in the following table:
\begin{table}[ht]
\centering 
\begin{threeparttable}
    \begin{tabular}{ccc}
    $p$ &  & $p_0$ \\
     \midrule\midrule
$C_{reg}$ & transverse & $\delta_p$\\
          & tangent & 0\\
$C_{mix}$ &  & 0\\ 
$C_\infty$ &  & 0 \\
$E_{reg}$ & transverse & $\delta_p$\\
          & tangent & 0\\
$E_{\infty}$ & & 0\\
\midrule\midrule
\end{tabular}
\end{threeparttable}
\caption{Limit of $q_t(p,\cdot)$ as $t\to 0$}
\end{table}

\newpage

\begin{proposition}\label{alpha}
Fix $\alpha>0$ and $k,l\in\mathbb{N}$, there exists constants $C,c>0$ depending on $\alpha,k,l$ such that if $|\sqrt{x}-\sqrt{x_1}|\geq\alpha>0$, $|y-y_1|\geq\alpha>0$ and $0\leq x,x_1,y,y_1\leq L$, then
\begin{gather*}
 |(\partial_x)^kp_d(t,x,x_1)|\leq Ce^{-c/t}x_1^{d-1}, \qquad |(\partial_y)^lk_t^e(y,y_1)|\leq Ce^{-c/t}.
\end{gather*}
\end{proposition}
\begin{proof}
For the first estimate, let $\lambda=\frac{x}{t}, w=\frac{x_1}{t}$. When $k=0$, 
\begin{align*}
    p_d(t,x,y)y^{1-d}=\frac{1}{t^d}e^{-(\lambda+w)}\psi_d(\lambda w).
\end{align*}
\begin{enumerate}
    \item If $\lambda w\leq 1$, then this term is bounded by
    \begin{align*}
        \frac{C}{t^d}e^{-(\lambda+w)}\leq \frac{C}{t^d} e^{-(\sqrt{\lambda}-\sqrt{w})^2}=\frac{C}{t^d}e^{-\frac{(\sqrt{x}-\sqrt{x_1})^2}{t}}.
    \end{align*}

    \item If $\lambda w\geq 1$, using the asymptotic expansion $\psi_d(z)\sim \frac{z^{\frac{1}{4}-\frac{d}{2}}e^{2\sqrt{z}}}{\sqrt{4\pi}}$,
    \begin{align*}
        \frac{1}{t^d}e^{-(\lambda+w)}\psi_d(\lambda w)\sim \frac{1}{\sqrt{4\pi}t^d}e^{-(\sqrt{\lambda}-\sqrt{w})^2}(\lambda w)^{\frac{1}{4}-\frac{d}{2}}.
    \end{align*}
    The right hand side is bounded by
    $\frac{1}{\sqrt{4\pi }t^d}e^{-(\sqrt{\lambda}-\sqrt{w})^2}$
    if $d\geq\frac{1}{2}$, and 
    bounded by 
    $\frac{L^{\frac{1}{2}-d}}{\sqrt{4\pi t}}e^{-(\sqrt{\lambda}-\sqrt{w})^2}$ if $0<d\leq\frac{1}{2}$.
    Overall, 
    \begin{align*}
        p_d(t,x,y)\leq\frac{C}{t^{d+\frac{1}{2}}}e^{-\frac{(\sqrt{x}-\sqrt{x_1})^2}{t}}\frac{1}{y^{d-1}}.
    \end{align*}
\end{enumerate}
The estimates for higher order derivatives are proved similarly.   

The second estimate is not hard to get since we can repeatedly use that for $\forall b\geq 0$, \[\frac{1}{t^b}e^{-\frac{(y-y_{1})^{2}}{4ct}}\leq\frac{C_b}{t^b}\frac{t^b}{|y-y_1|^{2b}}\leq C_{b,\alpha}.\]
\end{proof}

\subsection{Proof of Theorem \ref{heat kernel}} 
So far we have constructed the local heat kernel $q^{t}_{ij}(\cdot,q)$ for $\forall q\in\mathring{P}$. By the classical elliptic theory, there exists the Dirichlet heat kernel $q_t^U$ in $P_U$. We patch them together by defining the global kernel parametrix:
\begin{gather}\label{global_parametrix}
    q_{t}\left(d_{1},d_{2},l_{1},l_{2}\right)=\sum_{i=0}^4\sum_{j}{}\phi_{i,j,\epsilon}\left(d_{1},d_{2}\right)q_{t}^{ij}\left((\psi_{ij,\epsilon}(d_{1}),\psi_{ij,\epsilon}(d_{2}),\psi_{ij,\epsilon}(l_{1}),\psi_{ij,\epsilon}(l_{2})\right)\chi_{ij,\epsilon}\left(l_{1},l_{2}\right)\cdot|\text{det} \psi_{ij}(l_{1},l_{2})|\\
    +\psi q_t^U(d_1,d_2,l_1,l_2)(1-\varphi_U).
\end{gather}
Now set
\begin{align*}
    e_{t}(d_{1},d_{2}.l_{1},l_{2})=(\partial_{t}-L)q_{t}\left(d_{1},d_{2},l_{1},l_{2}\right).
\end{align*}
Again 
\begin{align*}
    e_{t}(d_{1},d_{2},l_{1},l_{2})=\left(\sum_{i=0}^{3}\underset{j}{\sum}\phi_{i,j,\epsilon}(\widetilde{L^{ij}}-L^{ij})q_{i,j}^{t}\chi_{i,j,\epsilon}\right)+\left(\sum_{i=0}^{3}\underset{j}{\sum}[\phi_{i,j,\epsilon},L]q_{i,j}^{t}\chi_{i,j,\epsilon}\right)+[\psi,L]q_t^U(1-\varphi_U).
\end{align*}

By construction, $e_{t}(d_{1},d_{2},l_{1},l_{2})$ is supported on $P_{U}\times P\times[0,\infty)$. To be more precise, the support of $e_{t}(d_{1},d_{2},l_{1},l_{2})$ is in an off-diagonal region: $d\left((d_{1},d_{2}),(l_{1},l_{2})\right)>\alpha$ for some $\alpha$, hence for $\forall T>0$, in $[0,T]$ it is bounded and by $Ce^{-\frac{C_\alpha}{t}}$ for some $C_\alpha>0$ using the estimate in Proposition \ref{alpha}. We let $A^te_t$ be the solution to the inhomogeneous problem
\begin{align*}
    (\partial_{t}-L)A^te_t(\cdot,l_1,l_2)=e_t(\cdot,l_1,l_2)\  \text{in}\ P\times[0,T]\ \text{with}\ A^0(\cdot, l_1, l_2)=0,
\end{align*}
Thus the global heat kernel is given by 
\begin{align*}
        H_t\left(d_1,d_2,l_{1},l_{2}\right)=q_{t}\left(d_{1},d_{2},l_{1},l_{2}\right)-A^te_t\left(d_1,d_2,l_{1},l_{2}\right).
\end{align*}

Next we investigate the regularity of the heat kernel $H_t$. If $L$ is Kimura operator, i.e., all the edges of $P$ are of Kimura type, as studied in \cite{EM}, Theorem \ref{heat_equation} can be extended to the case when $f\in C^{k,\gamma}(P), g\in C^{k,\gamma}(P\times [0,T])$, there exists a unique solution $w\in C^{k,2+\gamma}(P\times [0,T])$. For Cauchy problem, since $C^{k,2+\gamma}(P)\subset C^{k+1,\gamma}$(P), we can use bootstrap argument and the semi-group property of the solution to show that actually $w\in C^\infty(P\times (0,T])$. 

For $\forall t_0\in (0,T]$,
fix $q\in\mathring{P}$, the estimation in Proposition \ref{iteration} shows that the local heat kernel $q_{t_0}(\cdot, q)$ are in local $C^{0,\gamma}$ spaces, and the perturbation term
$A^te_t(\cdot,q)$ is also in $C^{0,\gamma}(P)$. Thus $H_{t_0}(\cdot,q)\in C^{0,\gamma}(P)$. We apply the regularity statement above to the Cauchy problem with initial condition  $H_{t_0}(\cdot, q)$, giving that $H_t(\cdot, q)\in C^\infty(P\times (t_0, T])$. Letting $t_0\to 0, T\to\infty$, then we have $(\cdot, t)\mapsto H_t(\cdot, q)\in C^\infty(P\times (0, \infty))$ for $\forall q\in\mathring{P}$ in the Kimura case.
 
When $L$ has a mixed type of boundary conditions, fix $q\in\mathring{P}$, we choose a neighborhood $U_Q$ of all the quadratic edge with $q\notin U_Q$ and $\chi\in C^\infty(U^c_Q)$ so that $\chi\equiv 1$ away from $U_Q$, $\chi(q)=1$. Let $K$ be a Kimura operator on $P$ so that the transverse/tangent boundary conditions align with $L$. We define the new operator 
\begin{gather*}
    \tilde{L}=\chi L+(1-\chi)K,
\end{gather*}
such defined $\tilde{L}$ is a Kimura operator, $\tilde{L}=L$ away from $U_ Q$, particularly $\tilde{L}(q)=L(q)$.
Denote the heat kernel of $\tilde{L}$ by $\tilde{H}_t$, we have 
\begin{gather}\label{h_t}
    (\partial_t-L)(H_t-\tilde{H}_t)(\cdot,q)=(L-\tilde{L})\tilde{H}_t(\cdot, q)\\
    \underset{t\to 0}{\lim}(H_t-\tilde{H}_t)(\cdot,q)=0.
\end{gather}
We have shown that $\tilde{H}_t\in C^\infty(P\times (0,\infty))$.
Since $L=\tilde{L}$ at $q$, the support of $(L-\tilde{L})\tilde{H}_t(\cdot, q)$ is away from $q$, so $(t,\cdot)\mapsto(L-\tilde{L})\tilde{H}_t(\cdot, q)$ is smooth and its high order derivatives are bounded by $e^{-\frac{C}{t}}$ for some $C>0$. Moreover since $q\in\mathring{P}$, when constructing the parametrix \ref{global_parametrix} of global heat kernel $H_t(\cdot,q), \tilde{H}_t(\cdot,q)$, we can choose the same Dirichlet heat kernel in a vicinity of $q$. Then the two remaining perturbation terms $A^te_t(\cdot, q), A^t\tilde{e}_t(\cdot, q)$ are both in $C^{k,\gamma}(P\times [0,T])$ for $k\in\mathbb{N}$. This roughly shows that $(H_t-\tilde{H}_t)(\cdot,q)\in C^{k,\gamma}(P\times [0,T])$.
Therefore by Theorem \ref{heat kernel} $H_t-\tilde{H}_t\in C^\infty(P\times [0,T])$. The argument above works for $\forall T>0$, so indeed we showed that $(\cdot, t)\mapsto H_t(\cdot, q)\in C^\infty(P\times (0,\infty))$.

Finally for the regularity of the forward variable, for $p\in P^{reg}$, $H_t(p,\cdot)$ is a solution to the Kolmogorov forward equation $(\partial_t-L^*)H_t(p,\cdot)=0$. Therefore by standard hypoellicity results for parabolic operators (\cite{stroock_2008} section 3.4.2), $H_t(p,\cdot)$ is smooth in $\mathring{P}$.

\section{Appendix} 

\subsection{Kimura kernel Estimates} 
We now provide several kernel estimates which are crucial in the construction of the local heat kernel at boundary points.
Recall that
\begin{gather*}
    p^d_t(x,y)=\left(\frac{y}{t}\right)^{d}e^{-\frac{x+y}{t}}\psi_{d}(\frac{xy}{t^{2}})\frac{1}{y},\ d>0,\\
    p^0_t(x,y)=e^{-\frac{x}{t}}\delta(y)+\left(\frac{x}{t}\right)e^{-\frac{x+y}{t}}\psi_2(\frac{xy}{t^{2}})\frac{1}{t},\ d=0.
\end{gather*}

\begin{lemma}\label{estimate_p}
There exists a constant $C_d>0$ uniformly bounded for $d\in [0,B]$ depending on $d$ such that 
\begin{gather}\label{bound-kernel}
    p_d(t,x,y)\leq\frac{C_d}{\sqrt{yt}}e^{-\frac{(\sqrt{x}-\sqrt{y})^2}{2t}},\ d\geq\frac{1}{2}\ \text{or}\  d=0, y\neq 0,\\
    p_d(t,x,y)\leq C_d\max\left(\frac{1}{t^d}e^{-\frac{(\sqrt{x}-\sqrt{y})^2}{t}}y^{d-1},\frac{1}{\sqrt{yt}}e^{-\frac{(\sqrt{x}-\sqrt{y})^2}{2t}}\right),\ d<\frac{1}{2}.
\end{gather}
\end{lemma}
\begin{proof}
Using the kernel formula,
\begin{align*}
    p^d_t(x,y)y=\left(\frac{y}{t}\right)^{d}e^{-\frac{x+y}{t}}\psi_{d}(\frac{xy}{t^{2}}).
\end{align*}
Let $\lambda=\frac{x}{t},w=\frac{y}{t}$.
First we treat the case when $d\geq\frac{1}{2}$ and want to show that
\begin{align}\label{formula}
    w^{d-\frac{1}{2}}e^{-(\lambda+w)}\psi_d(\lambda w)\leq Ce^{-\frac{(\sqrt{\lambda}-\sqrt{w})^2}{2}}.
\end{align}
\begin{enumerate}
    \item If $\lambda w\leq 1$, then $\psi_d(\lambda w)\leq C_d$ since $\psi_d$ is smooth and $e^{-(\lambda+w)}\leq e^{-(\sqrt{\lambda}-\sqrt{w})^2}$, so it remains to show that
    \begin{align}\label{d}
       w^{d-\frac{1}{2}}=O(e^{\frac{(\sqrt{\lambda}-\sqrt{w})^2}{2}}).
    \end{align}
If $w\leq 1$, $w^{d-\frac{1}{2}}\leq 1\leq e^{\frac{(\sqrt{\lambda}-\sqrt{w})^2}{2}}$. If $w>1$, then $w\leq \lambda$ since $\lambda w\leq 1$, so 
\[e^{\frac{(\sqrt{\lambda}-\sqrt{w})^2}{2}}\geq e ^{\frac{\left(\sqrt{\frac{1}{w}}-\sqrt{w}\right)^2}{2}}=\Omega(w^{d-\frac{1}{2}}).\]

\item If $\lambda w\geq 1$, using the asymptotic expansion $\psi_d(z)\sim \frac{z^{\frac{1}{4}-\frac{d}{2}}e^{2\sqrt{z}}}{\sqrt{4\pi}}$,
\begin{align*}
    w^{d-\frac{1}{2}}e^{-(\lambda+w)}\psi_d(\lambda w)=O\left(e^{-(\sqrt{\lambda}-\sqrt{w})^2}(\frac{w}{\lambda})^{\frac{d}{2}-\frac{1}{4}}\right).
\end{align*}
It remains to show that
\begin{gather}\label{frac_estimate}
    (\frac{w}{\lambda})^{\frac{d}{2}-\frac{1}{4}}=O(e^{\frac{(\sqrt{\lambda}-\sqrt{w})^2}{2}}).
\end{gather}
Let $k=\frac{w}{\lambda}$, then $\lambda\sqrt{k}\geq 1$,
\begin{align}\label{omega}
    e^{\frac{(\sqrt{\lambda}-\sqrt{w})^2}{2}}=e^{\frac{\lambda(\sqrt{k}-1)^2}{2}}\geq e^{\frac{\frac{1}{\sqrt{k}}(\sqrt{k}-1)^2}{2}}=\Omega(k^{\frac{d}{2}-\frac{1}{4}}).
\end{align}
\end{enumerate}
In the remaining case when $0< d\leq\frac{1}{2}$,
\begin{enumerate}
    \item If $\lambda w\leq 1$, $p_d(t,x,y)y^{1-d}t^d=e^{-(\lambda+w)}\psi_d(\lambda w)\leq\psi_d(1) e^{-(\sqrt{\lambda}-\sqrt{w})^2}$
    \item If $\lambda w\geq 1$, again we need to show that (\ref{formula}), for which 
we can apply the same estimate (\ref{omega}).
\end{enumerate}
In the last case when $d=0$, for $y\neq 0$,
\[\sqrt{yt}p_0(t,x,y)=\frac{x\sqrt{y}}{t^{\frac{3}{2}}}e^{-\frac{x+y}{t}}\psi_2(\frac{xy}{t^2})=\lambda\sqrt{w}e^{-(\lambda+w)}\psi_2(\lambda w).\]
\begin{enumerate}
    \item If $\lambda w\leq 1$, then the term above is bounded by 
    \begin{align*}
        C\sqrt{\lambda}e^{-(\lambda+w)}=O(e^{-\frac{(\sqrt{\lambda}-\sqrt{w})^2}{2}}).
    \end{align*}
    \item If $\lambda w\geq 1$, then
    \begin{align*}
        \lambda\sqrt{w}e^{-(\lambda+w)}\psi_2(\lambda w)\sim \left(\frac{\lambda}{w}\right)^{\frac{1}{4}}e^{-(\sqrt{\lambda}-\sqrt{w})^2}=O(e^{-\frac{(\sqrt{\lambda}-\sqrt{w})^2}{2}}).
    \end{align*}
\end{enumerate}
\end{proof}

\begin{lemma}
There exists a constant $C_d>0$ uniformly bounded for $d\in[0,B]$ such that 
\begin{gather*}
   \sqrt{x}\partial_xp_d(t,x,y) \leq\frac{C_d}{t\sqrt{y}}e^{-\frac{(\sqrt{x}-\sqrt{y})^2}{2t}},\quad d\geq\frac{1}{2}\ \text{or}\ d=0, y\neq 0,\\
   \sqrt{x}\partial_xp_d(t,x,y) \leq \frac{C_d}{\sqrt{t}}\max\left(\frac{1}{t^d}e^{-\frac{(\sqrt{x}-\sqrt{y})^2}{t}}y^{d-1},\frac{1}{\sqrt{yt}}e^{-\frac{(\sqrt{x}-\sqrt{y})^2}{2t}}\right),\quad d<\frac{1}{2}.
\end{gather*}
\end{lemma}
\begin{proof}
Let $\lambda=\frac{x}{t}, w=\frac{y}{t}$, then
\begin{align}
    \sqrt{xy}t\partial_{x}p_d(t,x,y)=\frac{\sqrt{x}}{\sqrt{y}}(\frac{y}{t})^de^{-\frac{x+y}{t}}\left|\frac{y}{t}\psi'_d(\frac{xy}{t^2})-\psi_d(\frac{xy}{t^2})\right|=\frac{\sqrt{\lambda}}{\sqrt{w}}w^de^{-(\lambda+w)}|w\psi'_d(\lambda w)-\psi_d(\lambda w)|.
\end{align}
We want to show that 
\begin{align*}
    \frac{\sqrt{\lambda}}{\sqrt{w}}w^de^{-(\lambda+w)}|w\psi'_d(\lambda w)-\psi_d(\lambda w)|=O (e^{\frac{(\sqrt{\lambda}-\sqrt{w})^2}{2}}).
\end{align*}
\begin{enumerate}
    \item If $\lambda w\leq 1$, it suffices to show that
    \[\frac{\sqrt{\lambda}}{\sqrt{w}}w^d=O (e^{\frac{(\sqrt{\lambda}-\sqrt{w})^2}{2}}).\]
    Let $k=\frac{\lambda}{w}$, then $w\leq\frac{1}{\sqrt{k}}$, fix $k$,
    \begin{align}\label{log}
    \text{ln}k+2d\text{ln}w-(\sqrt{k}-1)^2w
    \end{align}
    has maximal value at $min(\frac{1}{\sqrt{k}},\frac{2d}{(\sqrt{k}-1)^2})$. If $\frac{1}{\sqrt{k}}\leq\frac{2d}{(\sqrt{k}-1)^2}$, then $M\leq k\leq N$ for some $M,N>0$, so \[lnk+2dlnw-(\sqrt{k}-1)^2w\leq (1-d)lnk-\frac{(\sqrt{k}-1)^2}{\sqrt{k}}\]
    is bounded above. If $\frac{2d}{(\sqrt{k}-1)^2}\leq\frac{1}{\sqrt{k}}$, then 
    \[lnk+2dlnw-(\sqrt{k}-1)^2w\leq lnk-2dln(\sqrt{k}-1)^2+(2dln2d-2d).\]
    The right-hand side tends to $-\infty$ as $k\to 0$, and tends to $-\infty$ if $d>\frac{1}{2}$ or a finite number if $d=\frac{1}{2}$ as $k\to\infty$, so it is bounded above.
    \item If $\lambda w\geq 1$, then 
    \begin{align*}
        \frac{\sqrt{\lambda}}{\sqrt{w}}w^de^{-(\lambda+w)}|w\psi'_d(\lambda w)-\psi_d(\lambda w)|\sim (\frac{\lambda}{w})^{\frac{1}{4}-\frac{d}{2}}e^{-(\sqrt{\lambda}-\sqrt{w})^2}(\sqrt{w}-\sqrt{\lambda}).
    \end{align*}
\end{enumerate}
In general,
\begin{enumerate}
    \item If $\lambda w\leq 1$, $p_d(t,x,y)y^{1-d}t^{d+\frac{1}{2}}=\sqrt{\lambda}e^{-(\lambda+w)}|w\psi'_d(\lambda w)-\psi_d(\lambda w)|=O( e^{-\frac{(\sqrt{\lambda}-\sqrt{w})^2}{2}})$ by (\ref{d}).
    \item If $\lambda w\geq 1$, again we need to show that (\ref{formula}), for which 
we can apply the same estimate (\ref{omega})
\end{enumerate}
\end{proof}

\begin{lemma}\label{second_estimate}
There exists a constant $C_d>0$ uniformly bounded for $d\in [0,B]$ such that 
\begin{align*}
   x\partial^2_xp^d_t(x,y) & \leq\frac{C_d}{t\sqrt{yt}}e^{-\frac{(\sqrt{x}-\sqrt{y})^2}{2t}},\ d\geq\frac{1}{2}\ \text{or}\ d=0, y\neq 0,\\
   x\partial^2_xp^d_t(x,y) & \leq\frac{C_d}{t}\max\left(\frac{1}{t^d}e^{-\frac{(\sqrt{x}-\sqrt{y})^2}{t}}y^{d-1},\frac{1}{\sqrt{yt}}e^{-\frac{(\sqrt{x}-\sqrt{y})^2}{2t}}\right),\ d<\frac{1}{2}.
\end{align*}
\end{lemma}
\begin{proof}
Let $\lambda=\frac{x}{t}, w=\frac{y}{t}$, then
\begin{align}
    \sqrt{yt}tx\partial^2_{x}p_d(t,x,y)=w^{d-\frac{1}{2}}e^{-(\lambda+w)}[(\lambda+w)\psi_d(\lambda w)-2\lambda w\psi'_d(\lambda w)-dw\psi'_d(\lambda w)].
\end{align}
For $d>\frac{1}{2}$, we want to show that 
\begin{align*}
  w^{d-\frac{1}{2}}e^{-(\lambda+w)}[(\lambda+w)\psi_d(\lambda w)-2\lambda w\psi'_d(\lambda w)-dw\psi_d(\lambda w)]=O (e^{\frac{(\sqrt{\lambda}-\sqrt{w})^2}{2}}).
\end{align*}
\begin{enumerate}
    \item If $\lambda w\leq 1$, it suffices to show that
    \[(\lambda+w+1)w^{d-\frac{1}{2}}=O (e^{\frac{(\sqrt{\lambda}-\sqrt{w})^2}{2}}).\]
    
    \item If $\lambda w\geq 1$, then 
    \begin{align*}
        w^{d-\frac{1}{2}}e^{-(\lambda+w)}[(\lambda+w)\psi_d(\lambda w)-2\lambda w\psi'_d(\lambda w)-dw\psi_d(\lambda w)]\sim (\frac{\lambda}{w})^{\frac{1}{4}-\frac{d}{2}}e^{-(\sqrt{\lambda}-\sqrt{w})^2}[(\sqrt{w}-\sqrt{\lambda})^2+O(\frac{w+\lambda}{\sqrt{\lambda w}}+1)].
    \end{align*}
\end{enumerate}
Next we treat the case when $d<\frac{1}{2}$. 
\begin{enumerate}
    \item If $\lambda w\leq 1$, $p_d(t,x,y)y^{1-d}t^{d+\frac{1}{2}}=\sqrt{\lambda}e^{-(\lambda+w)}|w\psi'_d(\lambda w)-\psi_d(\lambda w)|=O( e^{-\frac{(\sqrt{\lambda}-\sqrt{w})^2}{2}})$ by (\ref{d}).
    \item If $\lambda w\geq 1$, again we need to show that (\ref{formula}), for which 
we can apply the same estimate (\ref{omega}).
\end{enumerate}
For $d=0$, we want to show that 
\begin{align*}
  2\lambda e^{-(\lambda+w)}\left[w^{\frac{3}{2}}\psi'_2(\lambda w)-\sqrt{w}\psi_2(\lambda w)\right]+\lambda^2e^{-(\lambda+w)}\left[w^{\frac{5}{2}}\psi''_2(\lambda w)-2w^{\frac{3}{2}}\psi'_2(\lambda w)+\sqrt{w}\psi_2(\lambda w)\right]=O (e^{\frac{(\sqrt{\lambda}-\sqrt{w})^2}{2}}).
\end{align*}
When $\lambda w\leq 1$, the left hand side is $O((\sqrt{\lambda}+\sqrt{w}+\lambda^{\frac{3}{2}})e^{-(\lambda+w)})=O (e^{\frac{(\sqrt{\lambda}-\sqrt{w})^2}{2}})$. When $\lambda w>1$, the left hand side has asymptotic 
\begin{align*}
    e^{-(\sqrt{\lambda}+\sqrt{w})^2}\left[\frac{1}{(\lambda w)^{\frac{1}{4}}}(\sqrt{w}-\sqrt{\lambda})+\left(\frac{\lambda}{w}(\sqrt{\lambda}-\sqrt{w})^2\right)^{\frac{1}{4}}\right]=O (e^{\frac{(\sqrt{\lambda}-\sqrt{w})^2}{2}}).
\end{align*}
\end{proof}

\begin{lemma}\label{first_estimate}
There exists a constant $C_d>0$ uniformly bounded for $d\in[0,B]$ such that 
\begin{gather}
    \partial_xp_d(t,x,y)  \leq\frac{C_d}{t\sqrt{yt}}e^{-\frac{(\sqrt{x}-\sqrt{y})^2}{2t}},\quad d\geq\frac{1}{2}\ \text{or}\ d=0, y\neq 0,\\
    \partial_xp_d(t,x,y)\leq C_d\max\left(\frac{1}{t^{d+1}}e^{-\frac{(\sqrt{x}-\sqrt{y})^2}{t}}y^{d-1},\frac{1}{t\sqrt{yt}}e^{-\frac{(\sqrt{x}-\sqrt{y})^2}{2t}}\right),\quad d<\frac{1}{2}.
\end{gather}
\end{lemma}
\begin{proof}
First compute
\begin{align*}
    yt\partial_{x}p_d(t,x,y)=(\frac{y}{t})^de^{-\frac{x+y}{t}}\left|\frac{y}{t}\psi'_d(\frac{xy}{t^2})-\psi_d(\frac{xy}{t^2})\right|.
\end{align*}
Let $\lambda=\frac{x}{t},w=\frac{y}{t}$. For the case $d\geq\frac{1}{2}$, we want to show that
\begin{align}\label{verify}
    w^{d-\frac{1}{2}}e^{-(\lambda+w)}\left|w\psi'_d(\lambda w)-\psi_d(\lambda w)\right|\leq Ce^{\frac{(\sqrt{\lambda}-\sqrt{w})^2}{2}}.
\end{align}
The proof is similar to the proof above, except that we now use the asymptotic expansion \[\psi_d(z)\sim \frac{z^{\frac{1}{4}-\frac{d}{2}}e^{2\sqrt{z}}}{\sqrt{4\pi}}, \psi'_d(z)\sim \frac{z^{-\frac{1}{4}-\frac{d}{2}}e^{2\sqrt{z}}}{\sqrt{4\pi}}\]
and $\psi'_d$ is also continuous at $0$. The case when $0\leq d\leq\frac{1}{2}$ can be proved similarly.
\end{proof}

\begin{lemma}\label{y_estimate}
For $k\in\mathbb{N}$, there exists a constant $C_d>0$ depending on $k$ uniformly bounded for $d\in[0,B]$ such that 
\begin{gather}
    (y\partial_y)^kp_d(t,x,y)  \leq\frac{C_d}{\sqrt{yt}}e^{-\frac{(\sqrt{x}-\sqrt{y})^2}{2t}},\quad d\geq\frac{1}{2}\ \text{or}\ d=0, y\neq 0,\\
   (y\partial_y)^kp_d(t,x,y)\leq C_d\max\left(\frac{1}{t^{d}}e^{-\frac{(\sqrt{x}-\sqrt{y})^2}{t}}y^{d-1},\frac{1}{\sqrt{yt}}e^{-\frac{(\sqrt{x}-\sqrt{y})^2}{2t}}\right), \quad d<\frac{1}{2}.
\end{gather}
\end{lemma}
\begin{proof}
Let $\lambda=\frac{x}{t},w=\frac{y}{t}$, then  
$(w\frac{\partial}{\partial w})^kp=(y\frac{\partial}{\partial y})^kp$,. We first consider  the case when $d\geq\frac{1}{2}$.
\begin{gather*}
    p^d_t(\lambda,w)=\frac{1}{t}w^{d-1}e^{-(\lambda+w)}\psi_{d}(\lambda w)=:I_0\\
    w\partial_wp^d_t(\lambda,w)=\frac{1}{t}w^de^{-(\lambda+w)}\left[\lambda\psi'_d(\lambda w)-\psi_{d}(\lambda w)\right]+(d-1)p^d_t(\lambda,w)=:I_{11}+(d-1)I_0\\
    (w\partial_w)^2p^d_t(\lambda,w)=\frac{1}{t}w^de^{-(\lambda+w)}\left[\lambda^2\psi''_d(\lambda w)-2\lambda\psi'_d(\lambda w)+\psi_{d}(\lambda w))\right]+(2d-1)I_1-d(d-1)I_0\\
    =:I_{21}+(2d-1)I_1-d(d-1)I_0
\end{gather*}
By induction, for $k\in\mathbb{N}$,
\begin{gather}
    (w\partial_w)^kp_t^d(x,y)=\frac{1}{t}w^de^{-(\lambda+w)}\left[\sum_{i=0}^k(-1)^{k-i}+
    \left(\begin{array}{c}
k\\
i
\end{array}\right)\lambda^i\psi_d^{(i)}(\lambda w)\right]+\sum_{j=0}^{k-1}c_j(w\partial_w)^{j}p^d_t(x,y)
\end{gather}
so we are left to show that
\begin{gather*}
    I:=w^{d-\frac{1}{2}}e^{-(\lambda+w)}\left[\sum_{i=0}^k(-1)^{k-i}\left(\begin{array}{c}
    k\\
    i
    \end{array}\right)\lambda^i\psi_d^{(i)}(\lambda w)\right]\leq Ce^{-\frac{(\sqrt{\lambda}-\sqrt{w})^2}{2}}
\end{gather*}
Use that
\begin{gather}
 \frac{d^k}{d^k_z}\psi_d=(\sqrt{z})^{1-d-k}I_{d-1+k}(2\sqrt{z})=\psi_{d+k}(z),\\
 \psi_d(z)\sim \frac{z^{\frac{1}{4}-\frac{d}{2}}e^{2\sqrt{z}}}{\sqrt{4\pi}}, z\to\infty
\end{gather}
\begin{enumerate}
    \item If $\lambda w\leq 1$, then $I\leq Cw^{d-\frac{1}{2}}\left(1+\lambda+...+\lambda^k\right)e^{- (\lambda+w)}=O(e^{-\frac{(\sqrt{\lambda}-\sqrt{w})^2}{2}})$ since $d\geq\frac{1}{2}$;\\
    \item If $\lambda w\geq 1$,
    \begin{gather*}
        I\sim\frac{1}{\sqrt{4\pi}} w^{d-\frac{1}{2}}e^{-(\sqrt{\lambda}-\sqrt{w})^2}\left[\sum_{i=0}^k(-1)^{k-i}\left(\begin{array}{c}
    k\\
    i
    \end{array}\right)\lambda^i(\lambda w)^{\frac{1}{4}-\frac{d+i}{2}}\right]\\
    \sim\frac{1}{\sqrt{4\pi}}\left(\frac{\lambda}{w}\right)^{\frac{1}{4}-\frac{d}{2}+\frac{k}{4}}\frac{1}{(\lambda w)^{\frac{k}{4}}}e^{-(\sqrt{\lambda}-\sqrt{w})^2}
    \end{gather*}
\end{enumerate}
so $I\leq\frac{1}{\sqrt{4\pi}}\left(\frac{\lambda}{w}\right)^{\frac{1}{4}-\frac{d}{2}+\frac{k}{4}}e^{-(\sqrt{\lambda}-\sqrt{w})^2}=O(e^{-\frac{(\sqrt{\lambda}-\sqrt{w})^2}{2}})$ by (\ref{frac_estimate}).

The case when $0\leq d\leq\frac{1}{2}$ can be proved similarly.
\end{proof}

\subsection{Heat kernel estimates}\label{heat kernel estimate}
Next we give kernel estimates for the heat kernel $k^{e}_{t}(x,x_{1})=\frac{1}{\sqrt{4\pi t}}e^{-\frac{(x-x_{1})^{2}}{4t}}$.

\begin{lemma}\label{gaussian1}
 For $c>0,\ 0<s<t$, there exists a constant $C>0$ such that
\begin{gather}
     \int_{-\infty}^{\infty}|\partial_{y}k^e_{t-s}(y,y_1)-\partial_{y}k^e_{t-s}(y',y_1)|\cdot| y-y_{1}|^{\gamma}\cdot e^{-\frac{(y_1-y_2)^2}{4cs}}dy_{1}\\
     \leq\frac{C}{\sqrt{t-s}}|y-y'|^{\gamma}\left(e^{-\frac{(y-y_2)^{2}}{4ct}}+e^{-\frac{(y'-y_2)^{2}}{4ct}}\right).
\end{gather}
\end{lemma}

For $x_1<x_2$, denote $\alpha_e=\frac{3x_1-x_2}{2},\ \beta_e=\frac{3x_2-x_1}{2}$, let $J=[\alpha_e, \beta_e]$.
\begin{lemma}\label{gaussian2}
For $0<\gamma<1$, there is a $C$ such that
\begin{align}
    \int_0^t\int_J|\partial_x^2k_{t-s}^e(y,y_1)||y-y_1|^\gamma e^{-\frac{(y_1-y_2)^2}{4s}} dy_1ds\leq C|y-y'|^\gamma e^{-\frac{(y-y_2)^2}{4t}},\\
    \int_0^t\int_J|\partial_x^2k_{t-s}^e(y',y_1)||y'-y_1|^\gamma e^{-\frac{(y_1-y_2)^2}{4s}} dyds\leq C|y-y'|^\gamma e^{-\frac{(y'-y_2)^2}{4t}}.
\end{align}
\end{lemma}

\begin{lemma}\label{gaussian3}
For $c>0,\ 0<s<t$, there exists a constant $C>0$ such that
\begin{gather*}
    \int_{0}^{t}\int_{J^c}|\partial_{y}^{2}k_{t-s}^{e}(y,y_{1})-\partial_{y'}^{2}k_{t-s}^{e}(y',y_{1})|\cdot|y-y_{1}|^{\gamma}\cdot e^{-\frac{(y_{1}-y_2)^2}{4s}}dy_{1}ds,\\
    \leq C| y-y'|^{\gamma}\left(e^{-\frac{(y-y_2)^{2}}{8t}}+e^{-\frac{(y'-y_2)^{2}}{8t}}\right).
\end{gather*}
\end{lemma}
\begin{proof}
These are all corollaries of \cite[Lemma 9.2.12]{EM}.
\end{proof}

\section*{Acknowledgment.} This research was partially supported by the U.S. National Science Foundation, Grants DMS-1908736 and EFMA-1641100. 



\bibliographystyle{plain}
\bibliography{ref}

\begin{thebibliography}{10}

\bibitem{GB}
Guillaume Bal.
\newblock Topological protection of perturbed edge states.
\newblock {\em Communications in Mathematical Sciences}, 17, 09 2017.

\bibitem{JC}
Jean Cerf.
\newblock Topologie de certains espaces de plongements.
\newblock {\em Bulletin de la Soci{\'e}t{\'e} Math{\'e}matique de France},
  89:227--380, 1961.

\bibitem{wright}
Linan Chen and Daniel~W. Stroock.
\newblock The fundamental solution to the wright–fisher equation.
\newblock {\em SIAM Journal on Mathematical Analysis}, 42(2):539--567, 2010.

\bibitem{CW19}
Linan Chen and Ian Weih-Wadman.
\newblock The fundamental solution to one-dimensional degenerate diffusion
  equation, i.
\newblock {\em Analysis of PDEs}, 2019.

\bibitem{AD}
Adrien Douady.
\newblock Variétés à bord anguleux et voisinages tubulaires.
\newblock {\em Séminaire Henri Cartan}, 14:1--11, 1961-1962.

\bibitem{FG}
F.~G. Dressel.
\newblock {The fundamental solution of the parabolic equation, II}.
\newblock {\em Duke Mathematical Journal}, 13(1):61 -- 70, 1946.

\bibitem{harnack}
Charles Epstein and Rafe Mazzeo.
\newblock Harnack inequalities and heat kernel estimates for degenerate
  diffusion operators arising in population biology.
\newblock {\em Applied Mathematics Research eXpress}, 2016, 06 2014.

\bibitem{EM}
Charles~L. Epstein and Rafe Mazzeo.
\newblock {\em Degenerate Diffusion Operators Arising in Population Biology}.
\newblock Princeton University Press, 2013.

\bibitem{epstein2019transition}
Charles~L Epstein and Camelia~A Pop.
\newblock Transition probabilities for degenerate diffusions arising in
  population genetics.
\newblock {\em Probability Theory and Related Fields}, 173(1):537--603, 2019.

\bibitem{Evans}
Lawrence~C. Evans.
\newblock {\em Partial differential equations}.
\newblock American Mathematical Society, Providence, R.I., 2010.

\bibitem{fisher_1931}
R.~A. Fisher.
\newblock Xvii.—the distribution of gene ratios for rare mutations.
\newblock {\em Proceedings of the Royal Society of Edinburgh}, 50:204–219,
  1931.

\bibitem{Fri}
Avner Friedman.
\newblock {\em Partial differential equations of parabolic type}.
\newblock Courier Dover Publications, 2008.

\bibitem{DJ}
Dominic Joyce.
\newblock On manifolds with corners.
\newblock {\em Advances in Geometric Analysis}, 21, 10 2009.

\bibitem{Kry}
Nicolai~V. Krylov.
\newblock {\em Lectures on Elliptic and Parabolic Equations in Holder Spaces}.
\newblock American Mathematical Soc., 1996.

\bibitem{RB93}
Richard Melrose.
\newblock {\em The Atiyah-Patodi-Singer Index Theorem (1st ed.)}.
\newblock A K Peters/CRC Press, 1993.

\bibitem{RB96}
Richard Melrose.
\newblock {\em Differential Analysis on Manifolds with Corners}.
\newblock unfinished book available at http://math.mit.edu/~rbm, 1996, 1993.

\bibitem{shimakura}
N.~Shimakura.
\newblock Formulas for diffusion approximations of some gene frequency models.
\newblock {\em J. Math. Kyoto Univ. 21}, no.1:19--45, 1981.

\bibitem{DS}
Daniel~W. Stroock.
\newblock {\em Partial Differential Equations for Probabilists}.
\newblock Cambridge Studies in Advanced Mathematics. Cambridge University
  Press, 2008.

\bibitem{stroock_2008}
Daniel~W. Stroock.
\newblock {\em Partial Differential Equations for Probabilists}.
\newblock Cambridge Studies in Advanced Mathematics. Cambridge University
  Press, 2008.

\bibitem{10.1093/genetics/16.2.97}
Sewall Wright.
\newblock {Evolution in mendelian polulations}.
\newblock {\em Genetics}, 16(2):97--159, 03 1931.

\end{thebibliography}


\end{document}